\documentclass{amsart}
\usepackage{amssymb,amsfonts,amsmath}
\usepackage[all]{xy}
\usepackage[shortlabels]{enumitem}
\usepackage{hyperref, color}
\usepackage{commath}
\usepackage{esdiff}	
\usepackage{graphicx,subfigure}	
\usepackage[center]{caption}

\newcommand{\C}{\mathbb C}
\newcommand{\R}{\mathbb R}

\newcommand{\h}{\mathbb H}
\newcommand{\s}{\mathbb S}
\newcommand{\M}{\mathbb M}
\newcommand{\im}[1]{\mathrm{Im}(#1)}
\newcommand{\iso}[1]{\mathrm{Iso}(#1)}

\DeclareMathOperator{\tr}{\mathrm{tr}}
\DeclareMathOperator{\divv}{\mathrm{div}}
\DeclareMathOperator{\arctanh}{\mathrm{atanh}}
\DeclareMathOperator{\arccosh}{\mathrm{acosh}}
\DeclareMathOperator{\arcsinh}{\mathrm{asinh}}
\DeclareMathOperator{\sech}{\mathrm{sech}}
\DeclareMathOperator{\Ric}{\mathrm{Ric}}
\DeclareMathOperator{\trace}{\mathrm{tr}}
\DeclareMathOperator{\psl2}{\widetilde{\mathrm{PSL}}_{2}(\R)}
\DeclareMathOperator{\ekappatau}{\mathbb{E}(\kappa,\tau)}

\theoremstyle{theorem}
\newtheorem{theorem}{Theorem}[section]
\newtheorem{lemma}[theorem]{Lemma}
\newtheorem{proposition}[theorem]{Proposition}
\newtheorem{corollary}[theorem]{Corollary}

\theoremstyle{definition}
\newtheorem{definition}[theorem]{Definition}
\newtheorem{example}[theorem]{Example}
\newtheorem{remark}[theorem]{Remark}
\newtheorem*{claim}{Claim}

\numberwithin{equation}{section}
\numberwithin{figure}{section}

\begin{document}

\title[CMC Isometric Immersions into $\s^2 \times \R$ and $\h^2 \times \R$]{Constant mean curvature Isometric Immersions into $\s^2 \times \R$ and $\h^2 \times \R$ and related results}

\author[Beno\^it Daniel]{Beno\^it Daniel$^1$}
\author[Iury Domingos]{Iury Domingos$^{1,2}$}
\author[Feliciano Vit\'orio]{Feliciano Vit\'orio$^2$}

\address{$^1$ Universit\'e de Lorraine \\ Institut \'Elie Cartan de Lorraine\\ UMR 7502 CNRS, B.P. 70239\\
F-54506 Vandoeuvre-l\`es-Nancy cedex, France.}

\address{$^2$ Universidade Federal de Alagoas \\ Instituto de Matem\'{a}tica  \\ Campus A. C. Sim\~{o}es, BR 104 - Norte, Km 97, 57072-970.
Macei\'o - AL, Brazil.}

\email{benoit.daniel@univ-lorraine.fr}

\email{domingos1@univ-lorraine.fr}

\email{feliciano@pos.mat.ufal.br}

\thanks{B. D. and F. V. were partially supported by the International Cooperation Program CAPES/COFECUB Fondation.
The work of I. D. was conducted during a scholarship supported by the International Cooperation Program CAPES/COFECUB Fondation at the University of Lorraine; financed by CAPES -- Brazilian Federal Agency for Support and Evaluation of Graduate Education within the Ministry of Education of Brazil.}

\maketitle

\begin{abstract}
In this article, we study constant mean curvature isometric immersions into $\s^2\times\R$ and $\h^2\times \R$ and we classify these isometric immersions when the surface has constant intrinsic curvature.  As applications, we use the sister surface correspondence to classify the constant mean curvature surfaces with constant intrinsic curvature in the $3-$dimensional homogenous manifolds $\mathbb{E}(\kappa, \tau)$ and we use the Torralbo-Urbano correspondence to classify the parallel mean curvature surfaces in $\s^2 \times \s^2$ and $\h^2 \times \h^2$ with constant intrinsic curvature. It is worthwhile to point out that these classifications provide new examples.
\end{abstract}

\section{Introduction}

This paper deals with the classification of constant mean curvature surfaces and parallel mean curvature surfaces with constant intrinsic curvature  in some ambient manifolds. In $\R^3$, it is a classical result that an $H$-constant mean curvature surface with constant intrinsic curvature $K$ is either part of a plane, when $H=0$, or part of a right circular cylinder with $K=0$ or part of a round sphere of radius $1/H$ with $K=H^2$, when $H\neq 0$ (Levi-Civita \cite{levi-civita37}).

In the 3-space forms $\M^3_c$, surfaces with $H$ and $K$ constants are isoparametric surfaces, that is, the two principal curvatures are constant. In this direction, the first results of local classification in $\s^3_c$ and $\h^3_c$ started with E. Cartan \cite{Cartan1938} and since then many related results have been obtained. For example, minimal surfaces with constant intrinsic curvature in 3-dimensional space forms $\mathbb{M}^3_{c}$ are either totally geodesic with $K=c$ or a part of the Clifford torus with $K=0$ in $\s^3_c$ (Chen \cite{Chen72}, Lawson \cite{lawson69}).

In codimension 2, for 4-dimensional space forms $\mathbb{M}^4_{c}$, minimal surfaces with constant intrinsic curvature are either totally geodesics with $K=c$, or a part of Clifford torus with $K=0$ in a totally geodesics $\s^3_c$ of $\s^4_c$, or a part of Veronese surfaces with $K=c/3$ in $\s^4_c$ (Kenmotsu \cite{kenmotsu83}). For parallel mean curvature surfaces with constant intrinsic curvature in the same ambient manifold either $K=0$ or $K=H^2+c$. If $c\geq0$, then the surface is either a part of product of circles with $K=0$ or a part of 2-sphere with $K=H^2+c$, where $H$ denotes the norm of the mean curvature vector (Hoffman \cite{hoffman1973}).
 When the ambient manifold is a complex $2-$dimensional space form ($\mathbb{C}^2$, $\mathbb{CH}^2(c)$, $\mathbb{CP}^2(c)$), parallel mean curvature surfaces with constant intrinsic curvature are either a product of circles, or a cylinder, or a round sphere, or a slant surface or one of the Hirakawa examples (Hirakawa  \cite{hirakawa06gd}).

 The aim of the work is to generalize these results to some other ambient homogeneous manifolds.

First, we classify the constant mean curvature surfaces with constant intrinsic curvature in $\s^2\times\R$ and $\h^2\times \R$ (Theorem \ref{theo-Kconstante}). The minimal case has been treated by the first author in \cite{danielimj15}: either $K=0$ or $K=c$ and the surface is totally geodesic, or $K=c<0$ and the surface is part of an associate surface of the parabolic generalized catenoid in $\h^2_c\times\R$. In this paper we consider the non-minimal case.  Our study is based on a system of partial differential equations satisfied by the metric, the \emph{angle} and \emph{height functions} of the surface (which holds, more generally, without the condition of constant intrinsic curvature).

We show that apart from the vertical cylinders with $K=0$, there are exactly two examples of $H$-constant mean curvature surfaces in $\h^2\times\R$ with constant intrinsic curvature. Both satisfy the relation $K=4H^2+c<0$. The first one is a surface found in the classification of surfaces whose Abresch-Rosenberg differential vanishes \cite{abresch2004} and the second one is a helicoidal surface that appears in the study of screw motion surfaces in $\s^2\times\R$ and $\h^2\times \R$, due to Sa Earp and Toubiana \cite{saearp2005}; however, it was not explicit there that this surface has both constant intrinsic and mean curvatures.

In particular, this classification provides an example of two CMC isometric immersions of the same Riemannian surface with a same mean curvature and which are not congruent up to an isometric reparametrization; this answers a question by Torralbo and Urbano \cite[Remark 1]{torralbotams12}.

As a corollary (Theorem \ref{theo-Ekappatau}), using the sister surface correspondence \cite{danielcmh07}, we obtain a classification of CMC surfaces with constant intrinsic curvature in the homogeneous $3$-manifolds $\mathbb{E}(\kappa,\tau)$, for $\tau \not = 0$ and $\kappa-4\tau^2 \not = 0$. We show that apart from the cases with constant angle functions, we get either a minimal surface in $\mathbb{E}(\kappa,\tau)$ invariant by parabolic isometries with $K=\kappa$ or one of twin helicoidal surfaces in $\mathbb{E}(\kappa,\tau)$ with $K=4H^2+\kappa<0$. The first one is a surface found by Pe\~nafiel in \cite{Penafiel2012} and the second ones are motivated by the study of screw motion surfaces in $\psl2$ in \cite{penafiel15} of the same author.

Next, as an application of the classification in $\M^2_c\times\R$ and the results by Torralbo and Urbano \cite{torralbotams12}, we obtain a classification of non-minimal parallel mean curvature (PMC) surfaces with constant intrinsic curvature in $\M^2_c\times\M^2_c$ (Theorems \ref{thms2s2} and \ref{thmh2h2}). Note that $\M^2_c\times\M^2_c$ are K\"ahler-Einstein manifolds and are among the 19 models of geometry in dimension 4 (see for instance \cite{Wall1986}). In particular, in $\h^2\times\h^2$ we obtain a new family of such surfaces; for some of them, the group of isometries induced by congruences (i.e., isometries of the ambient manifold) is discrete. Up to our knowledge, this also provides the first examples of PMC surfaces with non identically zero extrinsic normal curvature.

This work is organized as follows: In Section \ref{section2}, we fix some notations and we recall previous results about surfaces in  $\s^2\times\R$ and $\h^2\times \R$. Moreover, from of study the angle and height functions and compatibility equations obtained in \cite[Theorem 3.3]{danielcmh07}, we establish necessary and sufficient conditions for a Riemannian surface to be isometrically immersed as a constant mean curvature surface in $\s^2\times\R$ or $\h^2\times \R$ (Theorem \ref{theo1}). After, we derive an additional  order-1 equation using the classical Weitzenb\"ock-Bochner formula.

In Section \ref{section3}, we suppose that the metric of the surface has constant intrinsic curvature and we consider a function associated to the Abresch-Rosenberg differential. This function satisfies a new differential partial equation that can be reduced to an order-0 equation by the previous results. We get the classification for the $H$-constant mean curvature surfaces in $\M_c^2\times\R$ with constant intrinsic curvature (Theorem \ref{theo-Kconstante}).

In Section \ref{section4}, we give the first application of Theorem \ref{theo-Kconstante}. We use the sister surface correspondence \cite[Theorem 5.2]{danielcmh07} to get the classification of $H$-constant mean curvature surfaces in $\mathbb{E}(\kappa,\tau)$ with constant intrinsic curvature, for $\tau\neq 0$.

Finally, in Section \ref{sectionpmc}, we use Theorem \ref{theo-Kconstante} and the correspondence by Torralbo and Urbano \cite{torralbotams12} between PMC immersions into $\M^2_c\times\M^2_c$ and pairs of CMC immersions into $\M^2_c\times\R$ to classify PMC surfaces in $\M^2_c\times\M^2_c$ with constant intrinsic curvature. The main difficulty is that the Torralbo-Urbano correspondence relates immersions and not surfaces. 

\section{The angle and height functions}\label{section2}

We fix real numbers $c$ and $H$ such that $c\neq 0$. Let $\mathbb{M}^2_c$ be the simply connected Riemannian of constant intrinsic curvature $c$, that is, $\mathbb{M}^2_c=\mathbb{S}^2_c$ is the 2-sphere for $c>0$ and $\mathbb{M}^2_c=\mathbb{H}^2_c$ is the hyperbolic plane for $c<0$. For simplicity, we sometimes use the normalization $c=\pm 1$. In this case, $\s^2_{1}=\s^2$ and $\h^2_{-1}=\h^2$.

\begin{theorem}[Daniel, \cite{danieltams09}]\label{danieltams09}
Let $(\Sigma,\dif s^2)$ be an oriented simply connected Riemannian surface, $\nabla$ its Riemannian connection and $K$ be the intrinsic curvature of $\dif s^2$. Let $S:\mathrm{T}\Sigma\to \mathrm{T}\Sigma$ be a field of symmetric operators, $T\in \mathcal{X}(\Sigma)$ and $\nu:\Sigma\to[-1,1]$ be a smooth function. Then there is an isometric immersion $f:\Sigma\to\M^2_c\times \R$ such that the shape operator with respect to the normal $N$ associated to $f$ is
\[\dif f\circ S\circ \dif f^{-1}\]
and such that
\[\partial_t=\dif f(T)+\nu N\]
if and only if the 4-uple $(\dif s^2,S,T,\nu)$ satisfies the following equations on $\Sigma$:
\begin{gather}
K=\det S+c\nu^2,\tag{C1}\label{C1}\\
\nabla_X SY-\nabla_Y SX-S[X,Y]=c\nu\big(\langle Y,T\rangle X -\langle X,T\rangle Y\big),\tag{C2}\label{C2}\\
\nabla_X T=\nu SX,\tag{C3}\label{C3}\\
\dif\nu(X)+\langle SX, T\rangle=0,\tag{C4}\label{C4}\\
\|T\|^2+\nu^2=1\tag{C5}\label{C5}.
\end{gather}

Moreover the immersion is unique up to an isometry of $\M^2_c\times \R$ that preserve both orientations of $\M^2_c$ and $ \R$.
\end{theorem}

We will say that $(\dif s^2,S,T,\nu)$ are Gauss-Codazzi data of the immersion $f$ and that $\nu$ is its angle function. The height function of $f$ is the map $h=p\circ f$ where $p:\M^2_c\times\R\to\R$ is the projection onto the factor $\R$.

As a first result, we derive from Theorem \ref{danieltams09} a system of compatibility equations for $H$-constant mean curvature surfaces in $\M^2_c\times \R$ in terms of the angle and height functions. We will say that an isometry of $\M^2_c\times \R$ is horizontal if it is the identity on the factor $\R$.

\begin{theorem}\label{theo1}
Let $\Sigma$ be an $H$-constant mean curvature surface in $\M^2_c\times \R$.
Then the angle function $\nu: \Sigma\to [-1,1]$ and the height function $h:\Sigma\to \R$ of $\Sigma$ satisfy the following system
\begin{align}
  \|\nabla\nu +H \nabla h\|^2 &= (H^2-K+c\nu^2)(1-\nu^2),\label{M1} \\
  \Delta \nu &= (2K -c(1+\nu^2)-4H^2)\nu,\label{M2} \\
  \|\nabla h\|^2 &= 1-\nu^2,\label{M3} \\
  \Delta h &= 2H\nu,\label{M4}
\end{align}
where $K$ denotes the intrinsic  curvature of $\Sigma$.\\

Conversely, let $(\Sigma, \dif s^2)$ be an oriented simply connected Riemannian surface, with curvature $K$. Assume that
$\nu: \Sigma\to (-1,1)$ and $h:\Sigma\to\R$ are smooth functions satisfying the system of four partial
differential equations above. Then there is a
 $H$-constant mean curvature isometric immersion $f:\Sigma\to \M^2_c\times \R$ such that $\nu$ and $h$ are the angle function and the height function of $\Sigma$, respectively. Moreover the immersion is unique up to horizontal isometries of $\M^2_c\times \R$ that preserves the orientation of $\M^2_c$.
\end{theorem}

\begin{proof}
Let $\Sigma$ be an $H$-constant mean curvature surface in $\M^2_c\times \R$. Let $(\dif s^2,S,T,\nu)$ be the Gauss-Codazzi data on $\Sigma$
 and $J$ the rotation of angle $\pi/2$ on $\mathrm{T}\Sigma$. By the symmetry of $S$ and $\trace S=2H$, a straightforward computation shows that $SJ+JS=2H J$, and the compatibility equation \eqref{C1} implies that $\nabla \nu = - ST$.  Away from the points where $T=0$, we consider the orthonormal frame
$\{T/\|T\|,JT/\|T\|\}$. Then  $S$ has the form
\[S=\frac{1}{\|T\|^2}\begin{pmatrix}
  -\dif \nu(T) & -\dif \nu(JT) \\
  -\dif \nu(JT) & 2H\|T\|^2+\dif \nu(T)
\end{pmatrix}.\]

For the height function definition we have that
\[\dif  h(X)=\langle X,\partial_t\rangle,  \text{  for every $X\in \mathrm{T}\Sigma$},\]
that is, $T=\nabla h$ and then
\[\det S =-\frac{1}{\|T\|^2}\Big(\|\nabla \nu\|^2+
2H\langle\nabla\nu,\nabla h\rangle\Big).\]

Hence, the equations \eqref{C1} and \eqref{C5} imply that
\[(K-c\nu^2)(1-\nu^2)=-\|\nabla \nu\|^2-2H\langle\nabla\nu,\nabla h\rangle,\]
i.e.,  \eqref{M1} holds. When $\nu^2=1$, the right and left
sides of \eqref{M1} are equals to zero, then for this case \eqref{M1} also holds.

Let $L$ be the Jacobi operator of $\Sigma$, given by $L=\Delta + \|S\|^2 + \overline{\Ric}(N)$, where $\overline{\Ric}$ is the Ricci tensor of $\M^2_c\times \R$. Since
$\|S\|^2=4H^2-2\det S$ and $\overline{\Ric} (N)=c(1-\nu^2)$, the equation \eqref{C1} implies that
\[L=\Delta -2K +c(1+\nu^2)+4H^2.\]

On the other hand, as $\partial_t$ is a Killing field and $\nu=\langle \partial_t,N\rangle$  then
 $L\nu = 0$, so \eqref{M2} holds.

Since $\nabla h = T$, equation \eqref{M3} follows directly from equation \eqref{C5}. Moreover,
equation \eqref{C3} gives that $\nabla_X \nabla h = \nu S X$, which concludes \eqref{M4} by the definition of
divergence, finishing the first assertion. Note that also by equation \eqref{C3}, if $\nu^2\neq 1$, $T$ satisfies the following
equation for every $X\in \mathrm{T}\Sigma$:
\begin{equation}\label{A1}
\nabla_X T= -\frac{\nu}{1-\nu^2}\dif \nu(X)T
-\frac{\nu}{1-\nu^2}\Big(\dif \nu(JX)
+2H\langle T,JX\rangle\Big)JT.
\end{equation}

We now prove the second part of the theorem. Let $(\Sigma, \dif s^2)$ be a real analytic simply connected Riemannian surface, $\nu: \Sigma\to (-1,1)$ and $h:\Sigma\to\R$ smooth functions on $\Sigma$ satisfying equations \eqref{M1}, \eqref{M2}, \eqref{M3} and \eqref{M4}.

\begin{claim} The vector field $T=\nabla h\in \mathrm{T}\Sigma$ satisfies equation \eqref{A1}.\\

By equation \eqref{M3} and symmetry of the Hessian of $h$, that is $\langle\nabla_X\nabla h,Y\rangle=\langle\nabla_Y\nabla h,X\rangle$ for every $X, Y\in \mathrm{T}\Sigma$, we have
\[\nabla_{\nabla h}\nabla h = -\nu\nabla \nu.\]

Since $\nu^2<1$, considering the orthonormal frame $\{\nabla h /\|\nabla h\|,J\nabla h/\|\nabla h\|\}$, by the symmetry
of Hessian of $h$ we have
\[\nabla_X \nabla h =
-\frac{\nu}{1-\nu^2}\dif \nu(X)\nabla h
+\frac{1}{1-\nu^2}\langle\nabla_{J\nabla h}\nabla h,X\rangle J \nabla h\]
for every $X\in \mathrm{T}\Sigma$. On the other hand, again by the symmetry
of Hessian of $h$ we have
\[\nabla_{JX}\nabla h+J\nabla_{X}\nabla h=\Delta h JX, \text{  for every $X\in \mathrm{T}\Sigma.$}\]

Since equation \eqref{M4} holds, we obtain that
\[\nabla_X \nabla h =
-\frac{\nu}{1-\nu^2}\dif \nu(X)\nabla h
-\frac{\nu}{1-\nu^2}\Big(\dif \nu(JX)+2H\langle \nabla h,JX\rangle\Big)J\nabla h,\]
that is, the vector field $\nabla h$ satisfies equation \eqref{A1}.
\end{claim}

Since $T=\nabla h$ does not vanish on $\Sigma$, because $\nu^2<1$, there is a unique symmetric operator $S:\mathrm{T}\Sigma\to \mathrm{T}\Sigma$ with constant trace $2H$, such that $ST=-\nabla \nu$.

We affirm that the 4-uple $(\dif s^2,S,T,\nu)$ is Gauss-Codazzi data on $\Sigma$. To prove this, it is sufficient to show the equations \eqref{C1} and \eqref{C2}. This is because \eqref{C4} is follows of definition of $S$, \eqref{C5} is the same of \eqref{M3} and the \eqref{C3} follows of \eqref{A1}, when we write any $X\in \mathrm{T}\Sigma$ in the basis $\{\nabla h /\|\nabla h\|,J\nabla h/\|\nabla h\|\}$ and use $\nabla \nu=-ST$.\\

In a previous calculation, from the fact that $\nabla \nu=-ST$ and $S$ is a symmetric operator with $\tr S=2H$, we have shown that
\[\det S  = -\frac{1}{\|T\|^2}\Big(\|\nabla \nu\|^2+
2H\dif \nu(T)\Big).\]

Using equations \eqref{M1} and \eqref{M3} we obtain the Gauss equation \eqref{C1}, that is,
\[K=\det S+c\nu^2.\]

To show equation \eqref{C2}, it is sufficient to verify for $X=T$ and $Y=JT$. Since the rotation $J$ of angle $\pi/2$ on $\mathrm{T}\Sigma$ commutes with $\nabla_X$, for every $X\in \mathrm{T}\Sigma$, and $SJ+JS=2HJ$, by the equation \eqref{C3}
we get
\[[T,JT]=2\nu(JS-HJ)T.\]

By the symmetry of $S$ and $ST=-\nabla\nu$ we have
\[\langle S[T,JT],T\rangle=2H\nu\dif \nu(JT).\]

Furthermore, again since the rotation $J$ of angle $\pi/2$ on $\mathrm{T}\Sigma$ commutes with $\nabla_X$, $SJ+JS=2HJ$, $\nabla_T T=-\nu\nabla \nu$ and $ST=-\nabla\nu$, we get
\[\nabla_T SJT-\nabla_{JT}ST= -2H\nu J\nabla\nu +J \nabla_T \nabla \nu+\nabla_{JT}\nabla\nu,\]
and, by the symmetry of Hessian of $\nu$ and since $J^2=-\mathrm{I}$, we have
\[\langle\nabla_T SJT-\nabla_{JT}ST,T\rangle=2H\nu\dif \nu(JT),\]
then
\[\langle\nabla_T SJT-\nabla_{JT}ST-S[T,JT],T\rangle=0.\]

On the other hand, again by the symmetry of $S$, $SJ+JS=2HJ$, $ST=-\nabla\nu$ and the fact that $J$ is an isometry, we have
\[\langle S[T,JT],JT\rangle=-2\nu\big(\|\nabla \nu\|^2+3H\dif \nu(T)+2H^2\|T\|^2\big).\]

Moreover, since $J$ is an isometry
\[\langle\nabla_T SJT-\nabla_{JT}ST,JT\rangle=\|T\|^2\Delta \nu -2H\nu\dif \nu(T).\]

Then
\[\langle\nabla_T SJT-\nabla_{JT}ST-S[T,JT],JT\rangle=\|T\|^2\big(\Delta \nu +4H^2\nu\big) +2\nu\big(\|\nabla \nu\|^2+2H\dif \nu(T)\big)\]
and by equations \eqref{M1}, \eqref{M2} and \eqref{M3}, we obtain
\[\langle\nabla_T SJT-\nabla_{JT}ST-S[T,JT],JT\rangle=-c\nu(1-\nu^2)^2.\]

On the other hand, computing $c\nu\big(\langle Y,T\rangle X -\langle X,T\rangle Y\big)$ for $X=T$ and $Y=JT$, respectively, we get
\[c\nu\big\langle\langle T,JT\rangle T - \langle T,T\rangle JT, T\big\rangle=0\]
and
\[c\nu\big\langle\langle T,JT\rangle T - \langle T,T\rangle JT, JT\big\rangle=-c\nu(1-\nu^2)^2.\]

Thus we showed the Codazzi equation \eqref{C2} and then $(\dif s^2, S, T, \nu)$ are Gauss-Codazzi data on $U$. Since the operator $S$ and the vector field $T$ are determined in a unique way, the uniqueness follows from Theorem \ref{danieltams09} and the fact that the height function is prescribed.
\end{proof}

\begin{remark}
Assuming $\nu^2<1$ for the converse, it is possible to rewrite Theorem \ref{theo1} changing the gradient $\nabla h$ by vector field $T$ if the function $\nu$ and the a vector field $T$ satisfy
\begin{align*}
  \|\nabla\nu +H T\|^2 &= (H^2-K+c\nu^2)(1-\nu^2), \\
  \Delta \nu &= (2K -c(1+\nu^2)-4H^2)\nu, \\
  \|T\|^2 &= 1-\nu^2,\\
  \divv T &= 2H\nu,\\
  \langle\nabla_T T,JT\rangle&=\langle\nabla_{JT}T,T\rangle.
\end{align*}
This is because the last equation is equivalent to the fact that the $1-$form $\langle T, \cdot\rangle$ is closed and this is a necessary condition to show that $T$ satisfies equation \eqref{A1}.
\end{remark}

\begin{remark}
The minimal case has been studied by the first author in \cite{danielimj15}. He obtained a slightly different result, involving only $\nu$ and equations \eqref{M1} and \eqref{M2}.
\end{remark}

\begin{remark} \label{rmksign}
For a fixed $H$, the pair $(\nu,\nabla h)$ satisfies \eqref{M1}, \eqref{M2}, \eqref{M3} and \eqref{M4} for $H$ if and only if the pair $(-\nu,-\nabla h)$ does for $H$ and the pairs $(\nu,-\nabla h)$ and $(-\nu,\nabla h)$ do for $-H$. The isometric immersions corresponding to the pairs $(\nu,\nabla h)$ and $(-\nu,-\nabla h)$ are the same up to a $\pi$-rotation around a horizontal geodesic of $\M^2_c\times\R$. For the pair $(\nu,-\nabla h)$ (respectively, $(-\nu,\nabla h)$), its correspondent isometric immersion is the same of the isometric immersion correspond to $(\nu,\nabla h)$ up to an isometry of $\M^2_c\times\R$ that preserves the orientation of $\M^2_c$ and reverses the orientation of $\R$ (respectively, reverses the orientation of $\M^2_c$ and preserves the orientation of $\R$), see also \cite[Proposition 3.8]{danieltams09}.
\end{remark}

From Theorem \ref{theo1} follows the next result that characterizes $H$-constant mean surfaces in $\M^2_c\times \R$ with constant angle function. This result was already proved by Espinar and Rosenberg in \cite{RosenbergCMH11}.

Before that, we consider the smooth function $q:\Sigma\to\R$ introduced by Espinar and Rosenberg, in \cite{RosenbergCMH11}. This function $q$ is a normalization of the squared  norm of the Abresch-Rosenberg differential; it will play an important role in the case of $H$-constant mean curvature surface $\Sigma$ in $\M^2_c\times \R$ with constant intrinsic  curvature in the next section.

In \cite[Lemma 2.2]{RosenbergCMH11}, the authors show that
\[\|\nabla\nu\|^2=\frac{4H^2+c-c\nu^2}{4c}(4(H^2-K+c\nu^2)+c(1-\nu^2))-\frac{q}{c}.\]

Combining this relation and Theorem \ref{theo1}, we can see that $q$ satisfies
\begin{equation}\label{definition-q}
q=2Hc\langle \nabla \nu,\nabla h\rangle + 4H^2(H^2-K+c\nu^2) +2H^2c(1-\nu^2)+ \frac{c^2}{4}(1-\nu^2)^2.
\end{equation}

\begin{example}[Horizontal surfaces and vertical cylinders] The simplest examples of $H$-CMC surfaces in $\M^2_c\times\R$ are the horizontal surfaces and the vertical cylinders over curves of constant geodesic curvature.

Given $a\in \R$, then $\M^2_c\times\{a\}$ is a totally geodesic surface in $\M^2_c\times\R$ (then $H=0$) with intrinsic curvature $K=c$. The height function is constant and, since the normal vector $N$ of $\M^2_c\times\{a\}$ is parallel to $\partial_t$ and both are unitary vectors, the angle function $\nu$ satisfies $\nu^2=1$, in particular it is constant.

Let $\gamma\subset\M^2_c$ be a curve with constant geodesic curvature $k$. Then $\gamma\times\R$ is a $(k/2)$-CMC surface in $\M^2_c\times\R$ with intrinsic curvature $K=0$. Since the normal vector $N$ of $\gamma\times\R$ is orthogonal to $\partial_t$, then the angle function vanishes and the gradient of the height function is a principal direction of $\Sigma$ by equation \eqref{C4} and $2H=k$.

\end{example}

\begin{example}[ARL-surfaces] \label{exarlsurf}
Abresch and Rosenberg \cite{abresch2004} classified CMC surfaces with vanishing Abresch-Rosenberg differential. In particular, when $c<0$ they proved that, for each $H$ such that $0<4H^2<-c$, there exists a unique $H$-CMC surface $P_H$ in $\h^2_c\times\R$ invariant by parabolic isometries such that its Abresch-Rosenberg differential vanishes. Moreover, Leite \cite{Leite2007} proved that $P_H$ has constant intrinsic curvature $K=4H^2+c$. We note that, in the limit, $P_H$ is a horizontal surface of $\h^2\times \R$ when $H\to 0$ and a vertical cylinder of $\h^2\times \R$  when $4H^2\to -c$.

In this work, we say that $P_{H}$ is an Abresch-Rosenberg-Leite surface,  abbreviated  ARL-surface. In the Abresch-Rosenberg classification of surfaces with vanishing Abresch-Rosenberg differential, we can also see that ARL-surfaces are the only ones that the angle function $\nu$ is constant with $0<\nu^2<1$. More explicitly, the ARL-surfaces $P_H$ have the following properties:
\begin{itemize}
  \item The constant mean curvature $H$ satisfies $0<4H^2< -c$.
  \item $P_H$ has constant intrinsic curvature $K=4H^2+c$.
  \item The function $q$ vanishes identically on $P_H$.
  \item The function $\nu$ is constant on $P_H$, satisfying  $\nu^2=\frac{4H^2+c}{c}\in(0,1)$.
  \item $P_H$ is foliated by horizontal horocycles of principal curvature $2H$ orthogonally crossed by geodesics in $\h^2_c\times \R$.
\end{itemize}

In \cite{torralbotams12}, Torralbo and Urbano gave a conformal parametrization of the ARL-surfaces, found earlier by Leite in \cite{Leite2007}. Considering the hyperbolic plane model $\big((-\pi/2,\pi/2)\times\R,\frac{1}{-K\cos^2x}(\dif x^2+\dif y^2)\big)$,
in \cite[Example 4]{torralbotams12} it is shown that the height function is given by
\[h(x,y)=\frac{2H}{\sqrt{-K}}(y-\log\cos{x}).\]

For future computations, we consider the following change of coordinates:
\[(x,y)\mapsto z=\frac{i e^{-(y+ix)}}{\sqrt{-K}}.\]

This is a diffeomorphism between $\big(-\pi/2,\pi/2)\times\R$ and $\{z\in\C : \im{z}>0\}$. The metric $\frac{1}{-K\cos^2x}(\dif x^2+\dif y^2)$ is rewrite as
\[\dif s^2=\frac{4}{K(z-\overline{z})^2}|\dif z|^2,\]
and height function in the complex parameter $z$ is given by
\[h(z)=-\frac{2H}{\sqrt{-K}}\log\Big(\frac{\sqrt{-K}}{2i}(z-\overline{z})\Big).\]
\end{example}

\begin{corollary}[\cite{RosenbergCMH11}]\label{corollary1}  Let $\Sigma$ be an $H$-constant mean curvature surface in $\M^2_c\times \R$. If the angle function $\nu: \Sigma\to [-1,1]$ is constant then
\begin{itemize}
  \item either $\nu^2=1$, $K=c$, $H=0$ and $\Sigma$ is part of a horizontal surface $\M_c^2\times \{a\},$ for some $a\in \R$,
  \item or $\nu = 0$, $K=0$ and $\Sigma$ is part of a vertical cylinder $\gamma\times\R$, where $\gamma\subset \M_c^2$ is a curve of geodesic curvature $2H$,

  \item or $0<\nu^2< 1$, $K=4H^2+c <0$ and $\Sigma$ is part of an ARL-surface.
\end{itemize}
\end{corollary}

\begin{proof}
Assume that $\nu^2=1$. Then  \eqref{M2} implies $K=c+2H^2$ and \eqref{M3} implies that $T=0$, that is, the height function $h$ is constant on $\Sigma$, and so $H=0$, by \eqref{M4}. Hence $\Sigma$ is part of a horizontal surface $\M_c^2\times \{a\}$ of curvature $K=c$, for some $a\in \R$.

Assume that $\nu = 0$. Then equation \eqref{M1} implies that $K=0$. On the other hand, $\partial_t= T$, that is, the vertical vector is tangent to $\Sigma$ and it is a principal direction of $\Sigma$, by the equation \eqref{C4}. Then the other principal direction has eigenvalue $2H$, that is, $\Sigma$ is part of a vertical cylinder $\gamma\times \R$, where $\gamma\subset \M_c^2$ has geodesic curvature $2H$.

Assume that $0<\nu^2<1$. Then equations \eqref{M1} and \eqref{M3} imply $K=c\nu^2$. By \eqref{M2}, we obtain the following equation
\begin{equation}\label{eqcor1}
c(1-\nu^2)+4H^2=0.
\end{equation}

Note that there is no solution to \eqref{eqcor1} if $H=0$ or $c>0$. Then we have $c<0$ and $K=4H^2+c$. Moreover $c<K<0$ because $0<\nu^2<1$. Since $\nu$ constant and $K=c\nu^2$, again by equations \eqref{M1} and \eqref{eqcor1} we get that the function $q$ vanishes on $\Sigma$; therefore $\Sigma$ is part of an ARL-surface.
\end{proof}

The following result will be useful to future computations.

\begin{corollary}\label{corollary2}
Let $\Sigma$ be an $H$-constant mean curvature surface in $\M^2_c\times \R$. The function $\phi=\nu+Hh$ is constant if and only if $\Sigma$ is either part of a minimal horizontal surface $\M_c^2\times \{a\}$, for some $a\in \R$, or part of a minimal vertical cylinder $\gamma\times\R$, where $\gamma$ is a geodesic of $\M_c^2$.
\end{corollary}
\begin{proof}
If $\phi=\nu+Hh$ is a constant function on $\Sigma$ then $(H^2-K+c\nu^2)(1-\nu^2)=0$ and $\nu\big(2K-c(1+\nu^2)-2H^2\big)=0$.  Then from \eqref{M1}, \eqref{M2} and \eqref{M4} we get a system of two algebraic equations on $\nu$:
\begin{align*}
&(H^2-K+c\nu^2)(1-\nu^2)= 0,  \\
&\nu\big(-2(H^2-K+c\nu^2)-c(1-\nu^2)\big)=0.
\end{align*}

At the points where $0<\nu^2<1$, the first equation of the system above implies $H^2-K+c\nu^2=0$; then replacing this in the second one we have $-c(1-\nu^2)=0$, which cannot occur because $c\neq0$. Then $|\nu|$ does not take values in the interval $(0,1)$. Then, by connectedness and continuity, either $\nu^2=1$ or $\nu=0$ and then $\Sigma$ is part of a horizontal surface in $\M^2_c\times\R$ or $\Sigma$ is part of a vertical cylinder $\gamma\times\R$, where $\gamma\subset \M_c^2$, respectively.

For the minimality of $\Sigma$, note that if $\nu^2=1$, Corollary \ref{corollary1} implies that $K=c$ and $H=0$. If $\nu=0$, Corollary \ref{corollary1} implies that $K=0$, and by the first equation of the system above we get $H=0$.

Conversely, if $\Sigma$ is a minimal horizontal surface, then the unit normal $N$ is in the same direction of $\partial_t$. If $\Sigma$ is a minimal vertical cylinder over a geodesic of $\M^2_c$, then the unit normal $N$ is orthogonal to $\partial_t$. Then both cases imply that $\nu$ is constant on $\Sigma$. Since $\Sigma$ is minimal surface, $\phi$ is a constant function on $\Sigma$.
\end{proof}

Let $\Sigma$ be an $H$-constant mean curvature surface in $\M^2_c\times \R$ with  Gauss-Codazzi data $(\dif s^2, S, T, \nu)$. Consider a local orthonormal frame $\{e_1,e_2\}$ on $\Sigma$ such that $Je_1=e_2$ where $J$ is the rotation of angle $\pi/2$ on $\mathrm{T}\Sigma$.

Given a smooth function $f:\Sigma\to\R$ on $\Sigma$, we will set
\[f_i=\langle \nabla f,e_i \rangle
 \text{ \ and \ } f_{ij}=(\nabla^2 f)(e_i,e_j),\]
 where $\nabla^2 f$ is the symmetric Hessian 2-tensor of $f$. Then if $\nu$ and $h$ are the angle function and height function of $\Sigma$, respectively, setting $\phi=\nu+Hh$ we found the following system of two partial differential equations from Theorem \ref{theo1}:
\begin{align}
  \phi_1^2+\phi_2^2 &= -(1-\nu^2)(K-c\nu^2-H^2),\label{N1}  \\
  \phi_{11}+\phi_{22} &= \big(2K -c(1+\nu^2)-2H^2\big)\nu.\label{N2}
\end{align}

\begin{lemma}\label{lemma1}
The function $\phi$ satisfies
\begin{align}\label{E1}
  2\phi_{12}\|\nabla\phi\|^2 &= (1-\nu^2)(6c\nu\phi_1\phi_2-K_1\phi_2-K_2\phi_1) \\
       &\quad +2H\nu(H^2-K-c+2c\nu^2)(h_1\phi_2+h_2\phi_1)\nonumber
\end{align}
and
\begin{align}\label{E2}
     (\phi_{11}-\phi_{22})\|\nabla\phi\|^2 &= (1-\nu^2)\big(3c\nu(\phi_1^2-\phi_2^2)-K_1\phi_1+K_2\phi_2\big) \\
       &\quad +2H\nu(H^2-K-c+2c\nu^2)(h_1\phi_1-h_2\phi_2).\nonumber
\end{align}
\end{lemma}

\begin{proof}
Differentiating \eqref{N1} with respect to $e_1$ and  $e_2$, we have

\begin{align}
  2(\phi_1\phi_{11}+\phi_2\phi_{12}) &= 2\big(K-H^2+c(1-2\nu^2)\big)\nu\nu_1-(1-\nu^2)K_1\label{lemma1.1}, \\
  2(\phi_1\phi_{12}+\phi_2\phi_{22}) &= 2\big(K-H^2+c(1-2\nu^2)\big)\nu\nu_2-(1-\nu^2)K_2.\label{lemma1.2}
\end{align}

Making $\phi_2\eqref{lemma1.1}+\phi_1\eqref{lemma1.2}$ and using \eqref{N2} we obtain \eqref{E1}. In an analogous way, making $\phi_1\eqref{lemma1.1}-\phi_2\eqref{lemma1.2}$ we find
\begin{multline*}
  2(\phi^2_1\phi_{11}-\phi^2_2\phi_{22})= 2\big(K-H^2+c(1-2\nu^2)\big)\nu(\phi_1^2-\phi_2^2)-(1-\nu^2)(K_1\phi_1-K_2\phi_2)\\
  \quad +2H\nu\big(K-H^2+c(1-2\nu^2)\big)(\phi_1h_1-\phi_2h_2).
\end{multline*}

Since
\[2(\phi^2_1\phi_{11}-\phi^2_2\phi_{22})=(\Delta \phi)(\phi_1^2-\phi_2^2)+(\phi_{11}-\phi_{22})\|\nabla \phi\|^2,\]
 we get \eqref{E2}, using \eqref{N1}.
\end{proof}

\begin{proposition}\label{Prop.Bochner}
The functions $\nu$ and $h$ satisfy
\begin{gather}
     0 = (H^2-K+c\nu^2)\Delta K+\|\nabla K\|^2-6c\nu\langle\nabla K,\nabla \nu\rangle-2Hc\nu\langle\nabla K,\nabla h\rangle\nonumber\\
       \quad+6Hc(H^2-K-c\nu^2)\langle\nabla \nu,\nabla h\rangle+4H^2c\nu^2(H^2-K-2c+3c\nu^2)\label{M5}\\
       \quad-4(H^2-K+c\nu^2)(K-c-H^2)(K+2c\nu^2).\nonumber
\end{gather}
\end{proposition}

\begin{proof}
If $\phi$ is constant on a non empty open set of $\Sigma$, then by analyticity $\phi$ is constant on $\Sigma$. By Corollary \ref{corollary2}, $\nu$ and $K$ are constant functions on $\Sigma$, $H=0$ and $H^2-K+c\nu^2=0$, so \eqref{M5} holds.

From now on, we assume that $\phi$ is not a constant function on $\Sigma$. By analyticity, it is sufficient to prove equation \eqref{M5} on a non empty set $U\subset \Sigma$ on which $\nabla \phi$ does not vanish. Restricting $U$ if necessary, we can consider an orthonormal frame $\{e_1,e_2\}$ on $U$ and assume all the previous notations. The classical Weitzenböck-Bochner formula reads as
\begin{equation*}\label{bochner}
  \frac{1}{2}\Delta\|\nabla\phi\|^2=\langle\nabla\phi,\nabla\Delta \phi\rangle+\|\nabla^2\phi\|^2+K\|\nabla \phi\|^2.
\end{equation*}
Note that the Hessian term can be written
\[\|\nabla^2 \phi\|^2=\frac{1}{2}(\Delta \phi)^2+\frac{1}{2}(\phi_{11}-\phi_{22})^2+2\phi_{12}^2.\]

Then by Lemma \ref{lemma1} we have
\begin{align*}
2\|\nabla \phi\|^4\|\nabla^2\phi\|^2 &= \|\nabla \phi\|^4(\Delta\phi)^2\\
&\quad+(1-\nu^2)^2\|\nabla \phi\|^2\Big\{9c^2\nu^2\|\nabla \phi\|^2+\|\nabla K\|^2-6c\nu\langle\nabla K,\nabla \phi\rangle\Big\}\\
&\quad-4H\nu(1-\nu^2)\|\nabla \phi\|^2(H^2-K-c+2c\nu^2)\langle\nabla K,\nabla h\rangle\\
&\quad+12cH\nu^2(1-\nu^2)\|\nabla \phi\|^2(H^2-K-c+2c\nu^2)\langle\nabla \phi,\nabla h\rangle\\
&\quad+4H^2\nu^2(1-\nu^2)\|\nabla \phi\|^2(H^2-K-c+2c\nu^2)^2.
\end{align*}

Dividing the expression above by $(1-\nu^2)\|\nabla\phi\|^2$ and using \eqref{M4} we get
\begin{align}\label{WB-1}
2(H^2-K+c\nu^2)\|\nabla^2\phi\|^2 &= (H^2-K+c\nu^2)(\Delta\phi)^2\nonumber\\
&\quad+(1-\nu^2)\Big\{9c^2\nu^2\|\nabla \phi\|^2+\|\nabla K\|^2-6c\nu\langle\nabla K,\nabla \phi\rangle\Big\}\nonumber\\
&\quad-4H\nu(H^2-K-c+2c\nu^2)\langle\nabla K,\nabla h\rangle\\
&\quad+12cH\nu^2(H^2-K-c+2c\nu^2)\langle\nabla \nu,\nabla h\rangle\nonumber\\
&\quad+4H^2\nu^2(H^2-K-c+2c\nu^2)(H^2-K-c\nu^2+2c).\nonumber
\end{align}

By equation \eqref{N1}, and since $\Delta\nu^2=2(\nu\Delta\nu+\|\nabla\nu\|^2)$, we have
\begin{align*}
     -\frac{1}{2}\Delta\|\nabla \phi\|^2 &=\frac{1}{2}(1-\nu^2)\Delta K-2\nu\langle \nabla K,\nabla\nu\rangle\\
   &\quad+(H^2-K-c+2c\nu^2)\nu\Delta\nu+(H^2-K-c+6c\nu^2)\|\nabla\nu\|^2,
\end{align*}
and by equation \eqref{N2}
\begin{align*}
  \langle\nabla\phi,\nabla\Delta\phi\rangle &= 2\nu\langle\nabla K,\nabla\nu\rangle+2H\nu\langle\nabla K,\nabla h\rangle+(-2H^2+2K-c-3c\nu^2)\|\nabla\nu\|^2\\
  &\quad+H(-2H^2+2K-c-3c\nu^2)\langle\nabla\nu,\nabla h\rangle.
\end{align*}

Since equation \eqref{N1} implies $K\|\nabla\phi\|^2=K(H^2-K+c\nu^2)(1-\nu^2)$ and equation \eqref{M1} implies $\|\nabla\nu\|^2=(-K+c\nu^2)(1-\nu^2)-2H\langle\nabla\nu,\nabla h\rangle$, we get
\begin{align*}
  -\frac{1}{2}\Delta\|\nabla \phi\|^2+ \langle\nabla\phi,\nabla\Delta\phi\rangle+K\|\nabla\phi\|^2&=\frac{1}{2}(1-\nu^2)\Delta K+(H^2-K-c+2c\nu^2)\nu\Delta\nu \\
  &\quad+(-H^2+K-2c+3c\nu^2)(-K+c\nu^2)(1-\nu^2)\\
  &\quad+3Hc(1-3\nu^2)\langle\nabla\nu,\nabla h\rangle+2H\nu\langle\nabla K,\nabla h\rangle\\
  &\quad+K(H^2-K+c\nu^2)(1-\nu^2).
\end{align*}

Multiplying the expression above by $2(H^2-K+c\nu^2)$ and using \eqref{N1}, the Weitzenb\"ock-Bochner formula implies that
\begin{gather*}
0 =(1-\nu^2)\Big\{(H^2-K+c\nu^2)\Delta K+\|\nabla K\|^2-6c\nu\langle\nabla K,\nabla \phi\rangle+4Hc\nu\langle\nabla K,\nabla h\rangle\Big\}\\
  \quad +6Hc(1-\nu^2)(H^2-K-c\nu^2)\langle\nabla\nu,\nabla h\rangle+\|\nabla\phi\|^2\Big\{9c^2\nu^2(1-\nu^2)\\
  \quad+2(-K+c\nu^2)(-H^2+K-2c+3c\nu^2)+2K(H^2-K+c\nu^2) \Big\}\\
  \quad +(H^2-K+c\nu^2)(\Delta \phi)^2+2(H^2-K+c\nu^2)(H^2-K-c+2c\nu^2)\nu\Delta\nu\\
  \quad +4H^2\nu^2(H^2-K-c+2c\nu^2)(H^2-K-c\nu^2+2c).
\end{gather*}

Since $\nabla \phi=\nabla\nu +H\nabla h$, observing that $(H^2-K-c\nu^2+2c)=(H^2-K+c\nu^2)+2c(1-\nu^2)$, we get
\begin{gather*}
0 =(1-\nu^2)\Big\{(H^2-K+c\nu^2)\Delta K+\|\nabla K\|^2-6c\nu\langle\nabla K,\nabla \nu\rangle-2Hc\nu\langle\nabla K,\nabla h\rangle\Big\}\\
  +6Hc(1-\nu^2)(H^2-K-c\nu^2)\langle\nabla\nu,\nabla h\rangle+\|\nabla\phi\|^2\Big\{9c^2\nu^2(1-\nu^2)\\
  \quad+2(-K+c\nu^2)(-H^2+K-2c+3c\nu^2)+2K(H^2-K+c\nu^2)\\
  \quad +3c\nu^2(2H^2-2K+c+2c\nu^2) \Big\}+8H^2c\nu^2(1-\nu^2)(H^2-K-c+2c\nu^2).
\end{gather*}

Dividing by $(1-\nu^2)$ and using \eqref{N1} we get equation \eqref{M5}.
\end{proof}

\section{CMC surfaces with constant intrinsic  curvature}\label{section3}

Minimal surfaces in $\M^2_c\times\R$ with constant intrinsic curvature were classified in \cite[Theorem 4.2]{danielimj15}: such a surface is either totally geodesic or part of an associate surface of the parabolic generalized catenoid (a certain limit of catenoids). Regarding the non minimal case,
Corollary \ref{corollary1} provides some examples of $H$-constant mean curvature surfaces in $\M^2_c\times\R$ with constant intrinsic curvature, with $H\neq 0$. We will see next that these are not the only ones in $\h^2_c\times\R$.

This new example is based on the work \cite{saearp2005} by Sa Earp and Toubiana, where they study $H$-constant mean curvature screw motion surfaces in $\h^2\times\R$ and $\s^2\times\R$.

\begin{example}[Helicoidal surfaces in $\h^2_c\times\R$ satisfying $K=4H^2+c<0$]\label{helicoid} Up to scaling, suppose that $c=-1$. Consider the Poincar\'e disk model for $\h^2$. Let $K$ and $H$ be real numbers such that $H\neq 0$ and $K=4H^2-1<0$. Let $X:\R^2\to\h^2\times\R$ be the screw motion immersion given by
\begin{equation}\label{screwmotion-X}
X(\sigma,\tau)=\Big(\tanh\frac{\rho(\sigma)}{2} \cos\varphi(\sigma,\tau),\tanh\frac{\rho(\sigma)}{2}\sin\varphi(\sigma,\tau),\lambda(\sigma)+\varphi(\sigma,\tau)\Big),
\end{equation}
where the functions $\rho$, $\lambda$ and $\varphi$ are defined as
\begin{align*}
\rho(\sigma)&=\arccosh\Big(\frac{\cosh(\sqrt{-K}\sigma)}{\sqrt{-K}}\Big),\\
\lambda(\sigma)&= 2H\sigma+\arctan\Big(\frac{e^{2\sqrt{-K}\sigma}+2K+1}{4H\sqrt{-K}}\Big),\\
\varphi(\sigma,\tau)&=\sqrt{-K}\tau-\arctan\Big(\frac{e^{2\sqrt{-K}\sigma}+2K+1}{4H\sqrt{-K}}\Big).
\end{align*}

The height function is
\begin{align}
h(\sigma,\tau) &=\lambda(\sigma)+ \varphi(\sigma,\tau)\nonumber\\
  &=2H\sigma+\sqrt{-K}\tau.\label{h-helicoid}
\end{align}

We compute
\[\rho'(\sigma)=\frac{\sqrt{-K}\sinh(\sqrt{-K}\sigma)}{\sqrt{\cosh^2(\sqrt{-K}\sigma)+K}}\]
and
\[\lambda'(\sigma)=2H\frac{\cosh^2(\sqrt{-K\sigma})}{\cosh^2(\sqrt{-K\sigma})+K}.\]

By \cite[Proposition 9]{saearp2005} with $l=1$ we have that, choosing the appropiate orientation, the angle function is
\begin{align}
  \nu(\sigma) &= \frac{\sinh\rho(\sigma)}{\sqrt{1+\sinh^2\rho(\sigma)+(\lambda'(\sigma)/\rho'(\sigma))^2\sinh^2\rho(\sigma)}}\nonumber \\
   &= \sqrt{-K}\tanh(\sqrt{-K}\sigma)\label{nu-helicoid},
\end{align}
and we also compute
\[\frac{(\lambda'(\sigma)/\rho'(\sigma))\sinh^2\rho(\sigma)}{\sqrt{1+\sinh^2\rho(\sigma)+(\lambda'(\sigma)/\rho'(\sigma))^2\sinh^2\rho(\sigma)}}
=\frac{2H}{\sqrt{-K}}\cosh(\sqrt{-K}\sigma).\]

Hence the first formula of \cite[Lemma 11]{saearp2005} implies that $X(\R^2)$ has constant mean curvature $H$. Moreover, we see that equation $(\ast)$ in that lemma is satisfied for $d=0$ (see Figures 11 and 12 in \cite[Theorem 17]{saearp2005} for pictures of this surface).

Next, a straightforward computation shows that the induced metric $\dif s^2$ on $X(\R^2)$ is
\begin{equation}
\dif s^2=\dif\sigma^2+\cosh^2(\sqrt{-K}\sigma)\dif\tau^2\label{metric-helicoid}.
\end{equation}
By standard arguments, we obtain that this metric is complete and has intrinsic curvature $K$. Note that this surface when $H>0$ is the surface obtained in Theorem 19 in \cite{saearp2005} for $a=\sqrt{-K}$, $m=1/a$, $l=1$ and $U(\sigma)=\cosh(a\sigma)$.

Therefore, given $H\in\R$ satisfying $0<4H^2<1$, there is an $H$-constant mean curvature isometric immersion of $\R^2$ endowed with the metric \eqref{metric-helicoid} into $\h^2\times \R$, with constant intrinsic curvature $K=4H^2-1$, such that the angle and height functions are given by \eqref{nu-helicoid} and \eqref{h-helicoid}, respectively.

To find conformal coordinates, we consider the following change of coordinates on $\R^2$ for a complex parameter $z$:
\[(\sigma,\tau)\mapsto z=e^{-\sqrt{-K}\tau}\Big(\tanh(\sqrt{-K}\sigma)+i \sech(\sqrt{-K}\sigma)\Big).\]

This is a diffeomorphism between $\R^2$ and $\{z\in\C : \im{z}>0\}$. The metric \eqref{metric-helicoid} reads as
\[\dif s^2=\frac{4}{K(z-\overline{z})^2}|\dif z|^2.\]

We compute the angle and height functions above, \eqref{nu-helicoid} and \eqref{h-helicoid}, in the complex parameter by
\[\nu(z)=
\frac{\sqrt{-K}}{2|z|}(z+\overline{z})
\]
and
\[h(z)=\frac{2H}{\sqrt{-K}}\arcsinh\frac{i(z+\overline{z})}{z-\overline{z}}-\log|z|.\]
\end{example}

Let $\Sigma$ be an $H$-constant mean curvature surface in $\M^2_c\times \R$; we recall the smooth function $q:\Sigma\to\R_+$ defined on $\Sigma$ as before by
\[q=2Hc\langle \nabla \nu,\nabla h\rangle + 4H^2(H^2-K+c\nu^2) +2H^2c(1-\nu^2)+ \frac{c^2}{4}(1-\nu^2)^2.\]

Since the Abresch-Rosenberg differential is holomorphic, $q$ either has isolated zeroes or vanishes identically. Moreover, away from its zeroes, it is proved in \cite{RosenbergCMH11} that the function $q$ satisfies the following equation
\[\Delta \log q = 4K\]
i.e.,
\begin{equation}\label{Delta.log.q}
4Kq^2=q\Delta q-\|\nabla q\|^2.
\end{equation}
Therefore, this equation holds by continuity on the isolated zeroes of $q$, and also when $q$ vanishes identically.

If $H\neq 0$, when $\Sigma$ has constant mean curvature $H$ and also constant intrinsic  curvature $K$, Proposition \ref{Prop.Bochner} reads as
\[2Hcr\langle \nabla \nu,\nabla h\rangle+W=0\]
where
\[W=4H^2c\nu^2(H^2-K-2c+3c\nu^2)-4(H^2-K+c\nu^2)(K-c-H^2)(K+2c\nu^2)\]
and
\[r=3(H^2-K-c\nu^2).\]

Note that the condition $r= 0$ implies that $\nu$ is a constant function on $\Sigma$ and then $\Sigma$ is characterized by Corollary \ref{corollary1}. Then restricting in an open set of $\Sigma$ where $r\neq 0$, by Proposition \ref{Prop.Bochner} we can written the function $q$ as
\[q=\frac{p}{r}\]
where $p$ is defined by
\[p=-W+ r\Big(4H^2(H^2-K+c\nu^2) +2H^2c(1-\nu^2)+\frac{c^2}{4}(1-\nu^2)^2\Big).\]

Since $K$, $H$ and $c$ are constants, the functions $W$, $p$ and $r$ are polynomials of $\nu$ on $\Sigma$. In the next lemma, we transform equation \eqref{Delta.log.q}.

\begin{lemma}\label{Delta.log.beta}
Let $H\neq 0$, $\Sigma$ be an $H$-constant mean curvature surface in $\M^2_c\times \R$ with constant intrinsic  curvature $K$ and $U\subseteq \Sigma$ be an open set on which $r\neq 0$. If $K\neq 4H^2+c$, there is an even polynomial $g$ with degree $18$, such that $g\circ\nu=0$ on $U$.\break
\end{lemma}

\begin{proof}
In the open set $U\subseteq \Sigma$, since $rq=p$ we have that
\begin{align*}
r^4q\Delta q&= r^2p\Delta p-p^2r\Delta r -2p\langle r\nabla p -p\nabla r,\nabla r\rangle,\\
r^4\|\nabla q\|^2 &= r^2\|\nabla p\|^2-2p\langle r\nabla p - p\nabla r,\nabla r\rangle-p^2\|\nabla r\|^2.
\end{align*}

Subtracting the equations above, multiplying by $r$ and using equation \eqref{Delta.log.q} we get
\begin{equation}\label{M6}
4Kp^2r^3-p r^2(r\Delta p)+r^2(r\|\nabla p\|^2)-p^2(r\|\nabla r\|^2)+p^2r(r\Delta r)=0.\\
\end{equation}

Note that the quantities between parentheses in $\eqref{M6}$ are all polynomial of $\nu$. In fact, if $f\in \mathcal{C}^\infty (\Sigma)$ is an even polynomial of $\nu$ of degree at most 6, given by
\[f= f_0+f_2\nu^2+f_4\nu^4+f_6\nu^6,\]
where $f_i$ are constants, we have that
\begin{align*}
\|\nabla f\|^2  &=\nu^2(2f_2+4f_4\nu^2+6f_6\nu^4)^2\|\nabla\nu\|^2\\
\Delta f &=f_2\Delta\nu^2+f_4\Delta\nu^4+f_6\Delta\nu^6\\
&=(2f_2+4f_4\nu^2+6f_6\nu^4)\nu\Delta\nu+(2f_2+12f_4\nu^2+30f_6\nu^4)\|\nabla \nu\|^2.
\end{align*}
On the other hand, since $2Hcr\langle \nabla \nu,\nabla h\rangle+W=0$, by equation \eqref{M1} we have
\begin{align*}
 r\|\nabla\nu\|^2 &= r(-K+c\nu^2)(1-\nu^2)+\frac{1}{c}W\\
 &= \Big[-Kr_0+\frac{1}{c}W_0\Big]
 +\Big[(K+c)r_0-Kr_2+\frac{1}{c}W_2\Big]\nu^2\\
 &\quad+\Big[-cr_0+(K+c)r_2+\frac{1}{c}W_4\Big]\nu^4-cr_2\nu^6
\end{align*}
and
\[\nu\Delta\nu=(2K-4H^2-c)\nu^2-c\nu^4.\]

With these expressions we can see that equation \eqref{M6} has degree at most 20 in $\nu$, and now we proceeds to compute its coefficients of degrees 20 and 18.

For the term $4Kp^2r^3$, we have
\begin{align*}
  (4Kp^2r^3)_{20} &= 0, \\
  (4Kp^2r^3)_{18} &= 4Kp_6^2r_2^3.
\end{align*}
\indent\par
For the term $-p r^2(r\Delta p)$, we have
\begin{align*}
(-p r^2(r\Delta p))_{20}=& -6p_6^2r_2^3(\nu\Delta\nu)_4-30p_6^2r_2^2(r\|\nabla\nu\|^2)_6, \\
(-p r^2(r\Delta p))_{18}=&
-p_6r_2^2\big(6p_6r_0(\nu\Delta\nu)_4+4p_4r_2(\nu\Delta\nu)_4+6p_6r_2(\nu\Delta\nu)_2\\
&+12p_4(r\|\nabla\nu\|^2)_6+30p_6(r\|\nabla\nu\|^2)_4\big)\\
&-2p_6r_0r_2\big(6p_6r_2(\nu\Delta\nu)_4+30p_6(r\|\nabla\nu\|^2)_6\big)\\
&-p_4r^2_2\big(6p_6r_2(\nu\Delta\nu)_4+30p_6(r\|\nabla\nu\|^2)_6\big).
\end{align*}
\indent\par
For the term $r^2(r\|\nabla p\|^2)$, we have
\begin{align*}
(r^2(r\|\nabla p\|^2))_{20}&= 36p_6^2r_2^2 (r\|\nabla\nu\|^2)_6,\\
(r^2(r\|\nabla p\|^2))_{18}&= 72p_6^2r_0r_2 (r\|\nabla\nu\|^2)_6+r_2^2\big(36p_6^2 (r\|\nabla\nu\|^2)_4+48p_4p_6r_2^2 (r\|\nabla\nu\|^2)_6\big).
\end{align*}
\indent\par
For the term $-p^2(r\|\nabla r\|^2)$, we have
\begin{align*}
(-p^2(r\|\nabla r\|^2))_{20}=&-4p_6^2r_2^2(r\|\nabla\nu\|^2)_6,\\
(-p^2(r\|\nabla r\|^2))_{18}=&-4p_6^2r_2^2(r\|\nabla\nu\|^2)_4-8p_4p_6r_2^2(r\|\nabla\nu\|^2)_6.
\end{align*}
\indent\par
And finally, for the term $p^2r(r\Delta r)$, we have
\begin{align*}
(p^2r(r\Delta r))_{20} &= p_6^2r_2\big(2r_2^2(\nu\Delta\nu)_4+2r_2(r\|\nabla\nu\|^2)_6\big)\\
(p^2r(r\Delta r))_{18} &= p_6^2r_2\big(2r_0r_2(\nu\Delta\nu)_4+2r_2^2(\nu\Delta\nu)_2+2r_2(r\|\nabla\nu\|^2)_4\big)\\
&\quad  +p_6^2r_0\big(2r_2^2(\nu\Delta\nu)_4+2r_2(r\|\nabla\nu\|^2)_6\big)\\
&\quad  +2p_4p_6r_2\big(2r_2^2(\nu\Delta\nu)_4+2r_2(r\|\nabla\nu\|^2)_6\big).
\end{align*}
\indent\par
Summing all these terms of order 20 and 18, we get respectively
\[\eqref{M6}_{20}=4p_6^2r_2^2\big((r\|\nabla\nu\|^2)_6-r_2(\nu\Delta\nu)_4\big),\]
and
\begin{multline*}
\eqref{M6}_{18}=4Kp_6^2r_2^3-14p_6^2r_0r_2^2(\nu\Delta\nu)_4-6p_4p_6r_2^3(\nu\Delta\nu)_4-4p_6^2r_2^3(\nu\Delta\nu)_2\\
+14p_6^2r_0r_2(r\|\nabla\nu\|^2)_6+2p_4p_6r_2^2(r\|\nabla\nu\|^2)_6+4p_6^2r_2^2(r\|\nabla\nu\|^2)_4.
\end{multline*}

Since $p_4=-c^2(101H^2 - 29K + 26c)/4$, $p_6=-3c^3/4$, $r_0=3(H^2 - K)$, $r_2=-3c$, $(r\|\nabla\nu\|^2)_4=c(17H^2 - 8K + 5c)$, $(r\|\nabla\nu\|^2)_6=-cr_2$, $(\nu\Delta\nu)_2=- 4H^2 + 2K - c$ and $(\nu\Delta\nu)_4=-c$,  we have that $\eqref{M6}_{20}=0$ and \[\eqref{M6}_{18}=-486c^9(4H^2+c-K).\]

If $K\neq4H^2+c$ then \eqref{M6} implies that exists an even polynomial g of degree 18, such that $g\circ \nu=0$ on $U$.
\end{proof}

\begin{theorem}\label{theo-Kconstante} Let $H\neq 0$ and $\Sigma$ be an $H$-constant mean curvature surface in $\M^2_c\times \R$ with constant intrinsic  curvature $K$. Then one of the following holds:

\begin{itemize}
  \item either $K=0$ and $\Sigma$ is part of a vertical cylinder $\gamma\times\R$, where $\gamma\subset \M_c^2$ is a curve of geodesic curvature $2H$,
  \item or $c<0$, $K=4H^2+c<0$ and $\Sigma$ is part of either an ARL-surface or a surface of Example \ref{helicoid}.
\end{itemize}
\end{theorem}

\begin{proof}
Let $\nu:\Sigma \to [-1,1]$ be the angle function of $\Sigma$. If $r=0$ then $\nu$ is a constant function on $\Sigma$ and so $\Sigma$ is characterized by Corollary \ref{corollary1}. On an open set where $r\neq 0$, suppose that $K\neq4H^2+c$. Then by Lemma \ref{Delta.log.beta} there is an even polynomial $g$ such that $g\circ\nu=0$; then $\nu$ is a constant function on $\Sigma$ and again $\Sigma$ is characterized by Corollary \ref{corollary1}.

Suppose that $K=4H^2+c$. If $\nu$ is constant on $\Sigma$ then the result follows from Corollary \ref{corollary1} and $\Sigma$ is part of an ARL-surface. If $\nu$ is not a constant function on $\Sigma$, Proposition \ref{Prop.Bochner} implies that
\begin{equation}\label{gradnu.gradh}
c\langle \nabla\nu,\nabla h\rangle=2H(4H^2+c-c\nu^2)
\end{equation}
and by Theorem \ref{theo1} we obtain the following system
\begin{align}
\|\nabla \nu\|^2 &= -\frac{1}{c}(4H^2+c-c\nu^2)^2, \label{K=cte,gradnu}\\
  \Delta\nu &= (4H^2+c-c\nu^2)\nu.\label{K=cte,Deltanu}
\end{align}

The equation \eqref{K=cte,gradnu} implies that $c<0$. Using the Cauchy-Schwarz inequality for $\langle\nabla\nu,\nabla h\rangle$ in \eqref{gradnu.gradh}, since $r\neq 0$, by equations \eqref{M1} and \eqref{M3} we have that $4H^2+c-c\nu^2<0$. Consequently, $K=4H^2+c<0$ and $|\nu|<\sqrt{\frac{K}{c}}$.\\

The system \eqref{K=cte,Deltanu} and \eqref{K=cte,gradnu} implies that the function $\nu$ is isoparametric, that is, $\|\nabla\nu\|^2$ and $\Delta \nu$ are  functions of $\nu$. Then there is a local parametrization of $\Sigma$ such that $\nu$ is one of coordinates (see, e.g., \cite{kenmotsu00} and \cite[page 163]{eisenhart40}), i.e., there is local coordinates $(x_1,x_2)$  on $\Sigma$ such that $\nu(x_1,x_2)=x_1$, for $x_1\in I$, where $I\subseteq\Big(-\sqrt{\frac{K}{c}},\sqrt{\frac{K}{c}}\Big)$ is an open interval, and
\[\dif s^2=\frac{1}{F(x_1)^2}\dif x_1^2+G(x_1)^2\dif x^2_2,\]
for $F(x_1)=\|\nabla \nu\|$ and $G:I\to\R$ defined by
\[F(x_1)G(x_1)=\exp\Big(\int\frac{\Delta \nu}{\|\nabla \nu\|^2}\dif x_1\Big).\]

By \eqref{gradnu.gradh} we have
\[F(x_1)=\frac{1}{\sqrt{-c}}(cx_1^2-c-4H^2)\] and we compute that, up to multiplication by a positive constant, \[G(x_1)=\frac{\sqrt{-c}}{(cx_1^2-c-4H^2)^{1/2}}.\]

Let $\{\partial_{x_1},\partial_{x_2}\}$ be the coordinate fields of the local parametrization $(x_1,x_2)$ on $\Sigma$. In the basis $\{\partial_{x_1},\partial_{x_2}\}$ of $\mathrm{T}\Sigma$, the gradient of height function $h$ of $\Sigma$ is written as
\[\nabla h = \frac{\partial h}{\partial x_1}F^2\partial_{x_1}+\frac{\partial h}{\partial x_2}\frac{1}{G^2}\partial_{x_2},\]
and the following system holds:
\begin{align}
\|\nabla h\|^2 &= 1-\nu^2 \label{X-1},\\
c\langle\nabla \nu,\nabla h\rangle &= 2H(4H^2+c-cx_1^2).\label{X-3}
\end{align}

Since $\nabla \nu=F^2\partial_{x_1}$, equations \eqref{X-1} and \eqref{X-3} imply that
\[\frac{\partial h}{\partial x_1}=\frac{2H}{cx_1^2-c-4H^2} \ \text{ \ and \ } \ \frac{\partial h}{\partial x_2}=\varepsilon,\]
for $\varepsilon=\pm 1$ and then, up to the addition of a constant, we obtain
\[h(x_1,x_2)=\frac{2H}{\sqrt{cK}}\arctanh \Big(\sqrt{\frac{c}{K}}x_1\Big)+\varepsilon x_2.\]

Considering the following change of coordinates
\[(x_1,x_2)\mapsto (\sigma,\tau)=\Bigg(\frac{1}{\sqrt{cK}}\arctanh\Big(\sqrt{\frac{c}{K}}x_1\Big),\frac{\varepsilon x_2}{\sqrt{-K}}\Bigg),\]
we can see that the metric $\dif s^2$, the angle and height functions of $\Sigma$ are also given, respectively, by
\begin{align*}
\dif s^2&=-c\dif\sigma^2-c\cosh^2(\sqrt{cK}\sigma)\dif\tau^2,\\
\nu(\sigma)&= \sqrt{\frac{K}{c}}\tanh(\sqrt{cK}\sigma),\\
h(\sigma,\tau)&= 2H\sigma+\sqrt{-K}\tau,
\end{align*}
that is, up to scaling, $\Sigma$ is the helicoidal surface of Example \ref{helicoid}.
\end{proof}

\begin{remark}
Note that Theorem \ref{theo-Kconstante} together with \cite[Theorem 4.2]{danielimj15} give a complete classification of constant mean curvature surfaces in $\M^2_c\times\R$ with constant intrinsic curvature.
\end{remark}

\begin{remark}
 For a given $H\neq0$, the ARL-surface and the helicoidal surface of Example \ref{helicoid} are complete surfaces in $\h^2\times\R$ with the same non-zero constant mean curvature, which are intrinsically isometric but not congruent; up to our knowledge this is the first example of such a pair (see \cite[Remark 1]{torralbotams12}).
\end{remark}

\section{CMC surfaces with constant intrinsic curvature in $\mathbb{E}(\kappa,\tau)$}\label{section4}

\subsection{Preliminaries and first examples}

In this section, as an application of Theorem \ref{theo-Kconstante}, we classify constant mean curvature surfaces in $\mathbb{E}(\kappa,\tau)$, for $\kappa-4\tau^2\neq 0$, with constant intrinsic curvature. The manifold $\mathbb{E}(\kappa,\tau)$ is a 3-homogeneous space with
a 4-dimensional isometry group; it is a Riemannian fibration of bundle curvature $\tau$ over $\M^2_\kappa$. A  comprehensive literature about constant mean curvature surfaces in these manifolds has been developed in the past few decades. For more details, we refer to \cite{abresch05,danielcmh07,fernandezmira10,scott83}. These spaces are classified as follows:
\begin{itemize}
  \item When $\tau=0$, $\mathbb{E}(\kappa,0)$ is the product space $\M^2_\kappa\times \R$,
  \item When $\tau\neq0$ and $\kappa>0$, $\mathbb{E}(\kappa,\tau)$ is a Berger sphere,
  \item When $\tau\neq0$ and $\kappa=0$, $\mathbb{E}(0,\tau)$ is the Heisenberg group with a left invariant metric.
  \item When $\tau\neq0$ and $\kappa<0$, $\mathbb{E}(\kappa,\tau)$ is the universal cover of $\textrm{PSL}_{2}(\R)$ with a left invariant metric, and we denote by $\psl2$.
\end{itemize}

From now on, we will assume that $\tau\neq 0$. As already mentioned in the product case, some examples of $H$-constant mean curvature surfaces in $\mathbb{E}(\kappa,\tau)$
with constant intrinsic curvature $K$ appears when angle function
is constant \cite{RosenbergCMH11}. Note that, by the recent classification by Dom\'{\i}nguez-V\'azquez and Manzano in \cite{vazquezmanzano18}, these constant angle CMC surfaces turn out to be the only isoparametric surfaces in $\mathbb{E}(\kappa,\tau)$.

\begin{example}[Vertical cylinders in $\mathbb{E}(\kappa,\tau)$]
Let $\gamma\subset\M^2_\kappa$ be a curve with constant geodesic curvature $k$. If $\pi:\mathbb{E}(\kappa,\tau)\to\M^2_\kappa$ is the Riemannian fibration then $\pi^{-1}(\gamma)$ is a $(k/2)$-CMC surface in $\mathbb{E}(\kappa,\tau)$ with intrinsic curvature $K=0$. Since the normal vector $N$ of $\pi^{-1}(\gamma)$ is orthogonal to the unit Killing vector field $\xi$, the angle function vanishes.
\end{example}

\begin{example}[Generalized ARL-surfaces in $\mathbb{E}(\kappa,\tau)$ with $\kappa<0$]
In this work, we call generalized ARL-surface the $H$-constant mean curvature surfaces $P_{H,\kappa,\tau}$ in $\mathbb{E}(\kappa,\tau)$ such that:
\begin{itemize}
  \item $P_{H,\kappa,\tau}$ has constant intrinsic curvature $K=4H^2+\kappa<0$.
  \item The Abresch-Rosenberg differential vanishes identically on $P_{H,\kappa,\tau}$.
  \item The function $\nu$ is constant on $P_{H,\kappa,\tau}$ satisfying  $\nu^2=\frac{4H^2+\kappa}{\kappa-4\tau^2}\in(0,1)$.
\end{itemize}

These surfaces generalizes the ARL-surfaces and have also been studied by Verpoort \cite{Verpoort2014}. Also in \cite{vazquezmanzano18}, Dom\'{\i}nguez-V\'azquez and Manzano gave an explicit parametrization of $P_{H,\kappa,\tau}$ has an entire graph.

Moreover, in \cite{RosenbergCMH11} it is shown that generalized ARL-surfaces
are the only $H$-constant mean curvature surfaces in $\mathbb{E}(\kappa,\tau)$ such that the angle function $\nu$ is constant and satisfies $0<\nu^2<1$.
\end{example}

\subsection{New examples in $\psl2$}
\begin{example}[Minimal surfaces in $\psl2$ with $K<0$ satisfying $K=\kappa$ invariant by parabolic isometries]\label{parabolicsurface-psl}
Up to scaling, we suppose that $\kappa=-1$. Consider $\psl2=\{(x,y,t)\in\R^3 : y>0\}$ endowed with the metric
\[\frac{\dif x^2+\dif y^2}{y^2}+\Big(-\frac{2\tau}{y}\dif x+\dif t\Big)^2, \text{\ with \ } \tau\neq 0.\]

Let $\Omega\subset\R^2$ be the open set given by $\Omega=\R\times(0,1)$. Consider the immersion $X:\Omega \to \psl2$ given by
\begin{equation}\label{parabolic-X}
X(x,y)=\big(x,y,\sqrt{4\tau^2+1}\arcsin{y}\big).
\end{equation}

This surface is the one studied by Pe\~nafiel in \cite[Lemma 4.2]{Penafiel2012} with $d=1$ (the other examples, with $d\in\R^*$, in that lemma can be reduced to this one if we consider the isometries of $\psl2$ given by $F_1(x,y,t)=(x/d,y/d,t)$, for $d>0$, and  $F_2(x,y,t)=(x/d,-y/d,-t)$, for $d<0$, with a change of coordinates; for $d=0$, the surface is a generalized ARL-surface). Moreover, this surface is invariant by 1-parameter group of parabolic isometries of $\psl2$ and by Pe\~nafiel's results, $X(\Omega)$ is a minimal surface.

The induced metric on $X(\Omega)$ is
\begin{equation}\label{metric-parabolic-psl}
\dif s^2=\frac{4\tau^2+1}{y^2}\dif x^2-\frac{4\tau \sqrt{4\tau^2+1}}{y\sqrt{1-y^2}}\dif x \dif y+\frac{4\tau^2y^2+1}{y^2(1-y^2)}\dif y^2.
\end{equation}

Choosing the appropiate orientation, the angle function $\nu$ is
\begin{equation}\label{nu-parabolic-psl}
\nu(y)=\frac{\sqrt{1-y^2}}{\sqrt{4\tau^2+1}}.
\end{equation}

Since the coefficients of the first fundamental form of $X(\Omega)$ depend only of $y$, standard computations using the Christoffel symbols show that the intrinsic curvature $K$ of $\dif s^2$ is given by
\begin{equation}\label{curvature-parabolic-psl}
K(y) = \frac{1}{4(EG-F^2)^2}\Big\{E_y(EG)_y-2(EG-F^2)E_{yy}-2FE_yF_y\Big\}.
\end{equation}

We compute all terms involved in \eqref{curvature-parabolic-psl}:
\[EG-F^2 = \frac{4\tau^2+1}{y^4(1-y^2)}, \ E_y = -\frac{2(4\tau^2+1)}{y^3}, \ E_{yy}=\frac{6(4\tau^2+1)}{y^4},\]
\[F_y = -\frac{2\tau \sqrt{4\tau^2+1}(2y^2-1)\sqrt{1-y^2}}{y^2(1-y^2)^2}\]
and
\[(EG)_y = \frac{16\tau^2y^4(4\tau^2+1) -2y^2(16\tau^4-8\tau^2-3) -4(4\tau^2+1)}{y^5(1-y^2)^2}.\]

A straightforward computation using these expressions and \eqref{curvature-parabolic-psl} shows that $X(\Omega)$  has constant intrinsic curvature $K=-1$.

Therefore, $X$ is a minimal isometric immersion of $\Omega$ endowed with the metric \eqref{metric-parabolic-psl} into $\psl2$, with constant intrinsic curvature $K=-1$, such that the angle function is given by \eqref{nu-parabolic-psl}.
Considering the $\pi$-rotation around the geodesic $\{x=0, t=\frac{\pi}{2}\sqrt{4\tau^2+1}\}$, we get a complete embedded minimal surface invariant by parabolic isometries of $\psl2$ (see Figures 10 and 11 in \cite[Example 4.3]{Penafiel2012} for pictures of this surface).
\end{example}

\begin{example}[Helicoidal surfaces in $\psl2$ satisfying $K=4H^2+\kappa<0$]\label{helicoid-psl}
Since the relation $K=4H^2+\kappa<0$ is invariant by scaling the metric of $\psl2$, we may multiply this metric by $1/\sqrt{-\kappa}$ and so assume that $\kappa=-1$. Now we consider $\psl2=\{(x,y,t)\in\R^3 : x^2+y^2<1\}$ endowed with the metric
\[\lambda^2(\dif x^2+\dif y^2)+\Big(2\tau\frac{\lambda_y}{\lambda}\dif x-2\tau\frac{\lambda_x}{\lambda}\dif y+\dif t\Big)^2,\]
with $\lambda=2/(1-(x^2+y^2))$.

Let $K$ and $H$ be real numbers, such that $H>0$ and $K=4H^2-1<0$. Let $\varepsilon=\pm 1$. We set
\begin{align*}
  A &= \frac{H\sqrt{1-4H^2}}{\sqrt{H^2+\tau^2}}\big(\sqrt{4\tau^2+1}-2\varepsilon\tau\big), \\
  B &= -\frac{H\sqrt{1-4H^2}}{\sqrt{H^2+\tau^2}}\big(\sqrt{4\tau^2+1}+2\varepsilon\tau\big),
  \\
  C &= -\frac{\varepsilon\tau\sqrt{1-4H^2}}{2H\sqrt{H^2+\tau^2}}.
\end{align*}

We note that \[1-A^2=\frac{\big(\varepsilon\tau(1-4H^2)+2H^2\sqrt{4\tau^2+1}\big)^2}{H^2+\tau^2};\] hence $0<A\leq 1$, and $A=1$ if and only if $\varepsilon\tau<0$ and $\tau^2(1-8H^2)=4H^4$. Similarly,
\[1-B^2=\frac{\big(\varepsilon\tau(1-4H^2)-2H^2\sqrt{4\tau^2+1}\big)^2}{H^2+\tau^2};\]
hence $-1\leq B<0$, and $B=-1$ if and only if $\varepsilon\tau>0$ and $\tau^2(1-8H^2)=4H^4$.

Consider the screw motion invariant immersion $X_{\varepsilon}:\R^2\to \psl2$ given by
\begin{equation*}
X_\varepsilon(\sigma,\theta)=\Big(f(\sigma)\cos(\theta),f(\sigma)\sin(\theta),u_{\varepsilon}(\sigma)+\big(-2\tau+\varepsilon\sqrt{4\tau^2+1}\big)\theta\Big),
\end{equation*}
where $f:\R\to(-1,1)$ and $u_\varepsilon:\R\to\R$ are defined as follows.

If $\varepsilon\tau>0$ or if $\varepsilon\tau<0$ and $\tau^2(1-8H^2)\neq4H^4$, the function $f$ is defined as
\[f(\sigma)=\frac{\sqrt{\cosh(\sigma)-A}}{\sqrt{\cosh(\sigma)-B}}.\]

If $\varepsilon\tau<0$ and $\tau^2(1-8H^2)=4H^4$, the function $f$ is defined as
\[f(\sigma) =\frac{\tanh(\sigma/2)}{\sqrt{4H^2\tanh^2(\sigma/2)+1-4H^2}}.\]

If $\tau^2(1-8H^2)\neq 4H^4$, the function $u_\varepsilon$ is defined as
\begin{multline*}
u_\varepsilon(\sigma) = \frac{2H\sqrt{4\tau^2+1}}{\sqrt{1-4H^2}}\sigma\\
+\frac{2\sqrt{H^2+\tau^2}}{1-4H^2}\Bigg\{\frac{A(A-C)}{\sqrt{1-A^2}}\arctan\Bigg(\frac{\sqrt{1+A}}{\sqrt{1-A}}\tanh(\sigma/2)\Bigg)\\
-\frac{B(B-C)}{\sqrt{1-B^2}}\arctan\Bigg(\frac{\sqrt{1+B}}{\sqrt{1-B}}\tanh(\sigma/2)\Bigg)\Bigg\}.
\end{multline*}

If $\varepsilon\tau>0$ and $\tau^2(1-8H^2)=4H^4$, the function $u_\varepsilon$ is defined as
\[u_\varepsilon(\sigma) = \frac{2H\sqrt{1-4H^2}}{\sqrt{1-8H^2}}\sigma+\sqrt{1-8H^2}\arctan\Bigg(\frac{\sqrt{1-4H^2}}{2H}\tanh(\sigma/2)\Bigg).\]

If $\varepsilon\tau<0$ and $\tau^2(1-8H^2)=4H^4$, the function $u_\varepsilon$ is defined as
\begin{align*}
u_\varepsilon(\sigma) &= \frac{2H\sqrt{1-4H^2}}{\sqrt{1-8H^2}}\sigma-\sqrt{1-8H^2}\arctan\Bigg(\frac{2H}{\sqrt{1-4H^2}}\tanh(\sigma/2)\Bigg).
\end{align*}

We note that $f$ and $u_\varepsilon$ remain unchanged if we multiply both $\varepsilon$ and $\tau$ by $-1$. In all cases, $X_\varepsilon$ is analytic and
\[u_\varepsilon'(\sigma)=\frac{2H\sqrt{4\tau^2+1}}{\sqrt{1-4H^2}}\frac{\big(\cosh(\sigma)-C\big)\cosh(\sigma)}
{\big(\cosh(\sigma)-A\big)\big(\cosh(\sigma)-B\big)}.\]

This surface is the analytic continuation of a surface that belongs to the screw motion invariant family in $\psl2$ studied by Pe\~nafiel in \cite{penafiel15}. In fact, if we consider the change of coordinates given by
\[\sigma(\rho) = \arccosh\Bigg(\frac{H\sqrt{1-4H^2}}{\sqrt{H^2+\tau^2}}\big(\sqrt{4\tau^2+1}\cosh(\rho)-2\varepsilon\tau\big)\Bigg)\]
for $\rho>\arccosh\frac{2\tau\varepsilon H\sqrt{1-4H^2}+\sqrt{H^2+\tau^2}}{H\sqrt{1-4H^2}\sqrt{4\tau^2+1}}$, we obtain the screw motion invariant surface of \cite[Section 3.3]{penafiel15} with $$l=-2\tau+\varepsilon\sqrt{4\tau^2+1}\quad\textrm{and}\quad d=\frac{\varepsilon\tau(1-4H^2)}{H\sqrt{4\tau^2+1}}$$ ($l$ is the pitch if the screw motion). Indeed, we have
\[(u_\varepsilon\circ\sigma)'(\rho)=\frac{\big(2H^2\sqrt{4\tau^{2} + 1}\cosh(\rho) +\varepsilon(1-4H^{2}) \tau\big) \big(\sqrt{4\tau^{2} + 1} \cosh(\rho) - 2\tau\varepsilon\big)}{\sinh(\rho)\sqrt{(1-4H^{2})H^2 \big(\sqrt{4\tau^{2} + 1} \cosh(\rho) - 2\tau\varepsilon\big)^{2} - H^{2} - \tau^{2}}}.\]
and $f(\sigma(\rho))=\tanh(\rho/2)$. By Pe\~nafiel's results, $X_\varepsilon(\R^2)$ has constant mean curvature $H$.

The induced metric on $X_\varepsilon(\R^2)$ is
\[\dif s^2=E\dif\sigma^2+2F\dif\sigma \dif\theta+G\dif\theta^2\]
where the terms $E$, $F$ and $G$ are given by
\begin{align*}
E(\sigma) &= \rho'(\sigma)^2+u_\varepsilon'(\sigma)^2,\\
F(\sigma) &= \frac{u_\varepsilon'(\sigma)}{H\sqrt{1-4H^2}\sqrt{4\tau^2+1}}\big(\varepsilon H\sqrt{1-4H^2}-2\tau\sqrt{H^2+\tau^2}\cosh(\sigma)\big), \\
G(\sigma) &=\frac{H^2+\tau^2}{H^2(1-4H^2)}\cosh^2(\sigma),
\end{align*}
where $\rho(\sigma)=2\arctanh(f(\sigma))$.

Since the coefficients of the first fundamental form of $X_\varepsilon(\R^2)$ depend only of $\sigma$, standard computations using the Christoffel symbols show that the intrinsic curvature $K$ of $\dif s^2$ is given by
\begin{equation}\label{curvature-screw-motion-psl}
K(\sigma) = \frac{1}{2(EG-F^2)^2}\Bigg\{\frac{1}{2}(EG-F^2)_{\sigma}G_{\sigma}-(EG-F^2)G_{\sigma\sigma}\Bigg\}.
\end{equation}

We compute all terms involved in \eqref{curvature-screw-motion-psl}:
\begin{align*}
G_\sigma &= \frac{2(H^2+\tau^2)}{H^2(1-4H^2)}\cosh(\sigma)\sinh(\sigma),\\
G_{\sigma\sigma}  &=\frac{2(H^2+\tau^2)}{H^2(1-4H^2)}\big(2\cosh^2(\sigma)-1\big),\\
EG-F^2&= \frac{H^2+\tau^2}{H^2(1-4H^2)^2}\cosh^2(\sigma),\\
(EG-F^2)_\sigma &= \frac{2(H^2+\tau^2)}{H^2(1-4H^2)^2}\cosh(\sigma)\sinh(\sigma).
\end{align*}

A straightforward computation using these expressions and \eqref{curvature-screw-motion-psl} shows that $X_\varepsilon(\R^2)$  has constant intrinsic curvature $K=4H^2-1<0$.

We now study the behaviour of the generating curve $\Gamma_\varepsilon=\{X_\varepsilon(\sigma,0) : \sigma\in\R\}$. First, note that $u_\varepsilon(\sigma)\to +\infty$ when $\sigma\to+\infty$ and $u_\varepsilon$ is odd; hence $X_\varepsilon(\R^2)$ is a complete surface. We distinguish three cases in terms of $\varepsilon\tau$ and $H$ (see figures \ref{screwmotion-fig1}, \ref{screwmotion-fig2} and \ref{screwmotion-fig3}):

\begin{description}
  \item[Type I] If $\varepsilon\tau>0$ or if $\varepsilon\tau<0$ and $\tau^2(1-8H^2)<4H^4$, then $C<1$, and so $u_\varepsilon'>0$ and  $u_\varepsilon$ is strictly increasing on $(0,+\infty)$. Moreover, $f$ is even and $u_\varepsilon$ is odd, so $\Gamma_\varepsilon$ is invariant by the $\pi$-rotation around the geodesic $\{y=0,t=0\}$.
  \item[Type II] If $\varepsilon\tau<0$ and $\tau^2(1-8H^2)>4H^4$, then $C>1$, and so we can consider $\sigma_1=\arccosh(C)$. We have $u_\varepsilon'<0$ on $(0,\sigma_1)$ and $u_\varepsilon'>0$ on $(\sigma_1,+\infty)$, so $u_\varepsilon$ is strictly decreasing on $(0,\sigma_1)$ and $u_\varepsilon'$ is strictly increasing on $(\sigma_1,+\infty)$. Moreover, $f$ is even and $u_\varepsilon$ is odd, so $\Gamma_\varepsilon$ is invariant by the $\pi$-rotation around the geodesic $\{y=0,t=0\}$. Since $u_\varepsilon$ changes of sign, $\Gamma_\varepsilon$ is not an embedded curve.
  \item[Type III] If $\varepsilon\tau<0$ and $\tau^2(1-8H^2)=4H^4$, then $u_\varepsilon'>0$ and so $u_\varepsilon$ is strictly increasing on $(0,+\infty)$. Moreover, $f$ and $u_\varepsilon$ are odd, so $\Gamma_\varepsilon$ is invariant by the $\pi$-rotation around the geodesic $\{x=0,t=0\}$.
\end{description}

Therefore, given $H>0$ satisfying $0<4H^2<1$, there are two complete $H$-constant mean curvature isometric immersions into $\psl2$, with constant intrinsic curvature $K=4H^2-1<0$: $X_1(\R^2)$ and $X_{-1}(\R^2)$ (these two surfaces are not congruent since they are invariant by screw motions with different pitches $l$).

\begin{figure}[h!]
\centering
\captionsetup{justification=centering}
\caption{\bf Generating curve $\Gamma_\varepsilon$ with $\tau=\frac{1}{2}$ and $\varepsilon = 1$:}\label{screwmotion-fig1}
\subfigure[$H=\frac{1}{4}$ (Type I).]{\includegraphics[width=40mm]{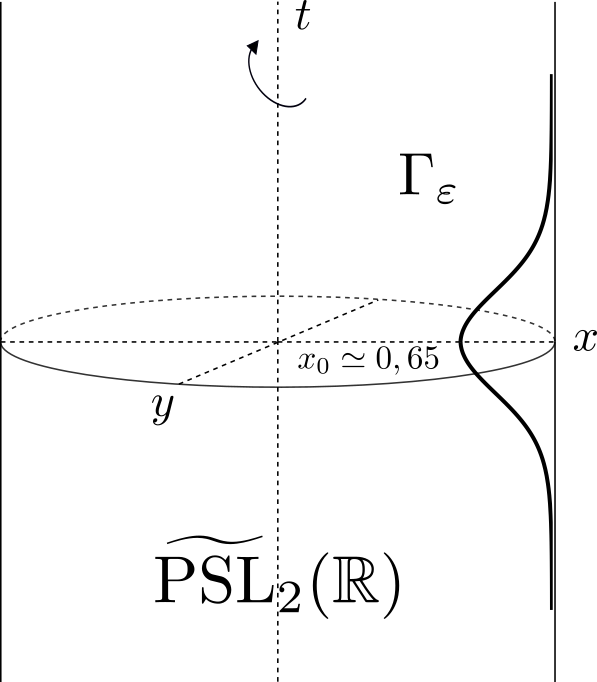}}\hfill
\subfigure[$H=\frac{\sqrt{\sqrt{2}-1}}{2}$ (Type I).]{\includegraphics[width=40mm]{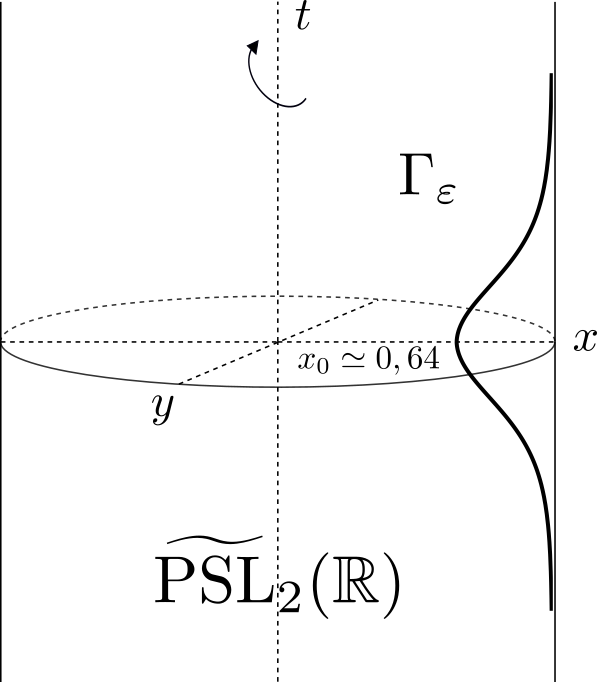}}\hfill
\subfigure[$H=\frac{\sqrt{2}}{3}$ (Type I).]{\includegraphics[width=40mm]{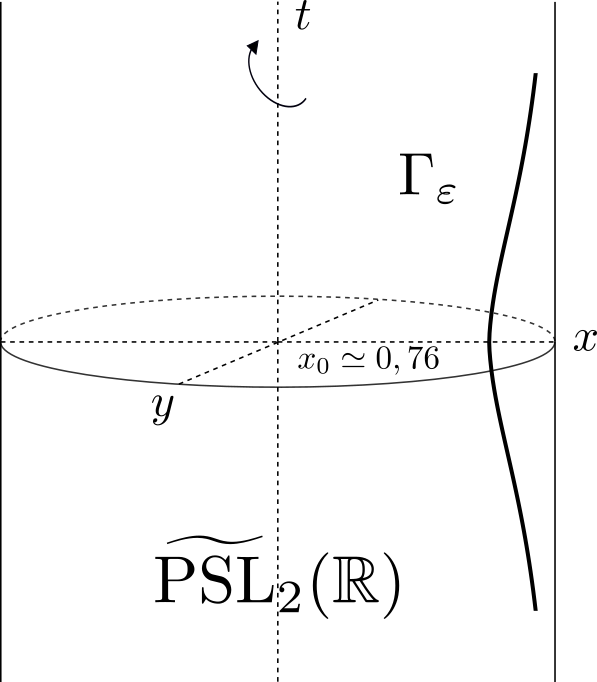}}
\end{figure}

\begin{figure}[h!]
\centering
\captionsetup{justification=centering}
\caption{\bf Generating curve $\Gamma_\varepsilon$ with $\tau=\frac{1}{2}$ and $\varepsilon = -1$:}\label{screwmotion-fig2}
\subfigure[$H=\frac{1}{4}$ (Type II).]{\includegraphics[width=40mm]{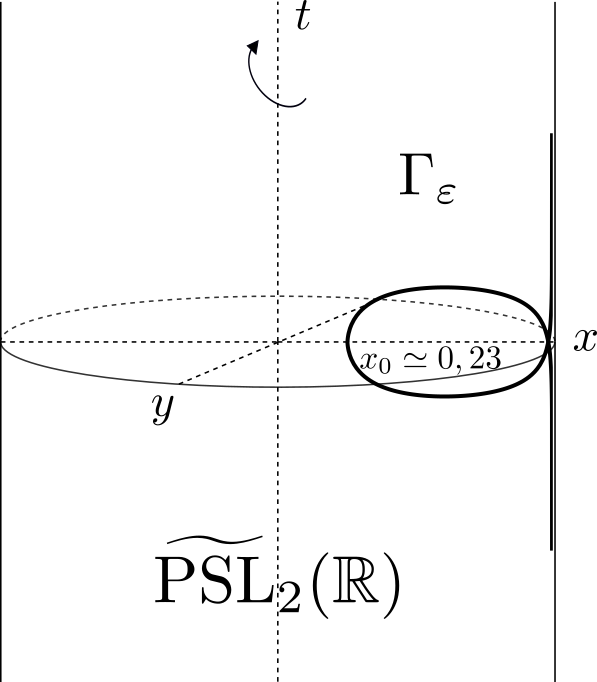}}\hfill
\subfigure[$H=\frac{\sqrt{\sqrt{2}-1}}{2}$ (Type III).]{\includegraphics[width=40mm]{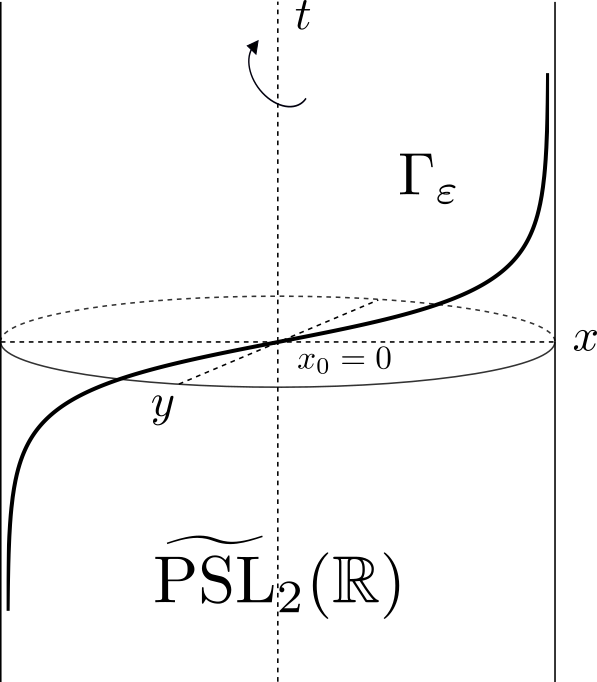}}\hfill
\subfigure[$H=\frac{\sqrt{2}}{3}$ (Type I).]{\includegraphics[width=40mm]{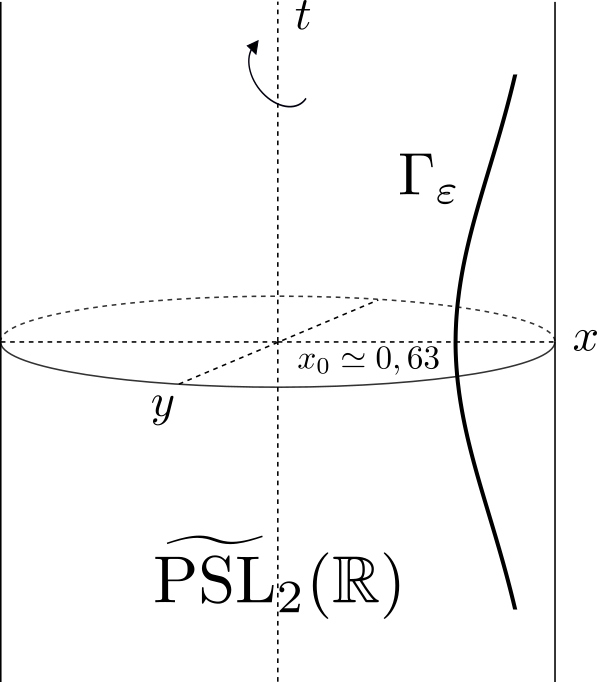}}
\end{figure}

\begin{figure}[h!]
\centering
\captionsetup{justification=centering}
\caption{\bf Complete screw motion surfaces in $\psl2$, with $\tau =\frac{1}{2}$, satisfying $K=4H^2-1<0$:}\label{screwmotion-fig3}
\subfigure[$\varepsilon = 1$ and $H=\frac{1}{4}$ (Type I).]{\includegraphics[width=50mm]{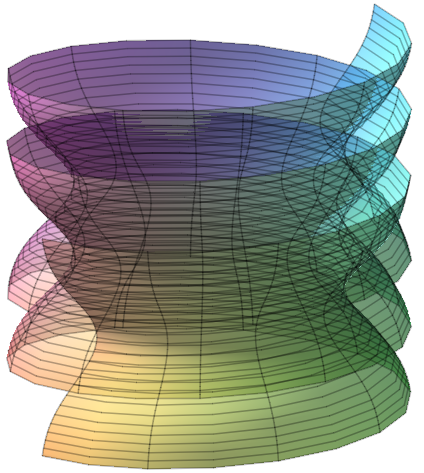}}\hfill
\subfigure[$\varepsilon =-1$ and $H=\frac{1}{4}$ (Type II).]{\includegraphics[width=50mm]{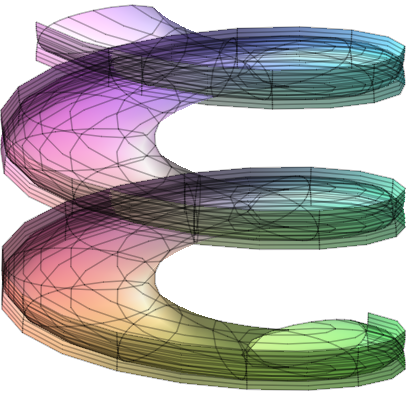}}\\
\subfigure[$\varepsilon =-1$ and $H=\frac{\sqrt{\sqrt{2}-1}}{2}$ (Type III).]{\includegraphics[width=50mm]{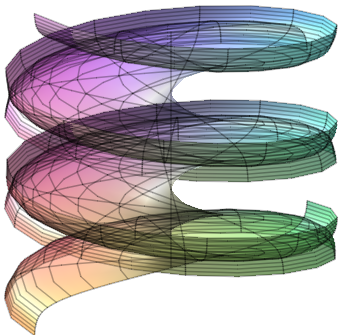}}
\end{figure}
\end{example}

\newpage

\begin{remark}
We explain how type I and II surfaces converge to a type III surface. We fix $\tau$ and $\varepsilon$ such that $\varepsilon\tau<0$ and we now index by $H$ the previous quantities and functions. Note that the $u_{\varepsilon,H}$ functions have been chosen so that $u_{\varepsilon,H}(0)=0$. Let $H_*$ be the mean curvature of the type III surface; it satisfies $\tau^2(1-8H_*^2)=4H_*^4$, hence $$H_*^2=-\frac\tau{2l}.$$

We note that $$A_H-C_H=\frac{\sqrt{1-4H^2}}{2H\sqrt{H^2+\tau^2}}\big(\varepsilon\tau(1-4H^2)+2H^2\sqrt{4\tau^2+1}\big),$$ hence $$1-A_H^2=\frac{4H^2}{1-4H^2}(A_H-C_H)^2,$$ $A_H-C_H>0$ for $H>H_*$ (type I surfaces) and $A_H-C_H<0$ for $H<H_*$ (type II surfaces).

Consequently, using the fact that $l(l+4\tau)=1$ and $1-8H_*^2=l^2$, when $H\to H_*$ we have $A_H\to1$, $B_H\to8H_*^2-1$, $C_H\to1$ and
$$\frac{2\sqrt{H^2+\tau^2}}{1-4H^2}\frac{A_H(A_H-C_H)}{\sqrt{1-A_H^2}}
\to\pm|l|.$$ So, since $\sqrt{\frac{1+A_H}{1-A_H}}\to+\infty$, we have $$u_{\varepsilon,H}\to u_{\varepsilon,H_*}\pm|l|\frac\pi2$$ uniformly on compact sets of $(0,+\infty)$. Here, the $+$ signs occur when the limit is for $H>H_*$ (type I surfaces) and the $-$ signs for $H<H_*$ (type II surfaces). On the other hand, $$f_H\to f_{H_*}$$ uniformly on compact sets of $(0,+\infty)$.

Considering the symmetries, we conclude that, when $H>H_*$, the curve $\Gamma_{\varepsilon,H}$ converges to $$T(\Gamma_{\varepsilon,H_*}^+)\cup S\cup (R\circ T)(\Gamma_{\varepsilon,H_*}^+)$$ where $\Gamma_{\varepsilon,H_*}^+=\{X_{\varepsilon,H_*}(\sigma,0):\sigma\in(0,+\infty)\}$, $S$ is the vertical segment $\{(0,0)\}\times[-|l|\pi/2,|l|\pi/2]$,
$T$ is the vertical translation of length $|l|\pi/2$ upward and $R$ the $\pi$-rotation around the geodesic $\{y=0,t=0\}$. And when $H<H_*$, the curve $\Gamma_{\varepsilon,H}$ converges to $$T^{-1}(\Gamma_{\varepsilon,H_*}^+)\cup S\cup (R\circ T^{-1})(\Gamma_{\varepsilon,H_*}^+).$$ In both cases, this limit set generates by screw motion of pitch $l$ a type III surface translated by $\pi|l|/2$ or $-\pi|l|/2$ in the vertical direction, i.e., a type III surface rotated by angle $\pi/2$ or $-\pi/2$, depending also on the sign of $l$.
\end{remark}

\subsection{Classification}
\begin{theorem}\label{theo-Ekappatau}
Let $\kappa$ and $\tau$ be real numbers such that $\tau\neq0$ and $\kappa-4\tau^2\neq 0$, and $\Sigma$ be an $H$-constant mean curvature surface in $\mathbb{E}(\kappa,\tau)$ with constant intrinsic curvature $K$. Then one of the following holds:
\begin{itemize}
  \item either $K=0$ and $\Sigma$ is part of a vertical cylinder over a curve $\gamma\subset \M_\kappa^2$ with geodesic curvature $2H$,
  \item or $\kappa<0$, $K=4H^2+\kappa<0$ and $\Sigma$ is part of a generalized ARL-surface,
  \item or $\kappa<0$, $H=0$, $K=\kappa$ and $\Sigma$ is part of a minimal surface of Example \ref{parabolicsurface-psl},
  \item or $\kappa<0$, $H\neq 0$, $K=4H^2+\kappa<0$ and $\Sigma$ is part of one of twin helicoidal surfaces of Example \ref{helicoid-psl}.
\end{itemize}
\end{theorem}

\begin{proof}
Let $\nu:\Sigma\to[-1,1]$ be the angle function of $\Sigma$. If $\nu$ is a constant function, since $\tau\neq 0$ by \cite[Theorem 2.2]{RosenbergCMH11} we get that either $\nu=0$, $K=0$ and $\Sigma$ is part of a vertical cylinder over a complete curve $\gamma\in\M^2_\kappa$ with geodesic curvature $2H$, or $0<\nu^2<1$, $K=4H^2+\kappa<0$ and $\Sigma$ is part of a generalized ARL-surface.

From now on, suppose $\nu$ is not a constant function. Let $(\dif s^2, S, T, \nu)$ be the Gauss-Codazzi data on $\Sigma$ into $\mathbb{E}(\kappa,\tau)$. Consider $\overline{H},\overline{\kappa}\in\R$ such that $\overline{H}^2=H^2+\tau^2$,  $\overline{\kappa}=\kappa-4\tau^2$ and $\overline{\tau}=0$. By the sister surface correspondence \cite[Theorem 5.2]{danielcmh07} the 4-uple $(\dif s^2, \overline{S}, \overline{T}, \nu)$ is the Gauss-Codazzi data of surface $\overline{\Sigma}$ in $\mathbb{E}(\overline{\kappa},0)=\M_{\overline{\kappa}}^2\times\R$, where
\[\overline{S}=e^{\theta J}(S-HI)+\overline{H}I \ \text{ \ and \ } \ \overline{T}=e^{\theta J}T,\]
with $\theta\in\R$ given by $i\overline{H}=e^{i\theta}(\tau+iH)$ and therefore
$\overline{\Sigma}$ has constant mean curvature $\overline{H}$ and is isometric to $\Sigma$. Since $\dif s^2$ has constant intrinsic curvature and $\nu$ is not a constant function, by Theorem \ref{theo-Kconstante} we get that
$\overline{\kappa}<0$ and $K=4\overline{H}^2+\overline{\kappa}<0$, that is, $\kappa-4\tau^2<0$ and $K=4H^2+\kappa<0$, and $\overline{\Sigma}$ is part of the helicoidal surface of Example \ref{helicoid} in $\M^2_{\kappa-4\tau^2}\times\R$.

However, the helicoidal surface of Example \ref{helicoid-psl} in $\M^2_{\kappa-4\tau^2}\times\R$ of constant mean curvature $\overline{H}$ has at most two constant mean curvature sister surfaces $\Sigma_1$ and $\Sigma_2$ in $\ekappatau$: denote by $H_k$ and $(\dif s^2,S_k, T_k, \nu)$, with $k=1,2$, their respective mean curvature and Gauss-Codazzi data, then $\overline{H}^2=H_k^2+\tau^2$, i.e., $H_1=-H_2=\pm H$, and
\[S_k=e^{\theta_k J}(\overline{S}-\overline{H}\mathrm{I})+H_i\mathrm{I} \ \text{ \ and \ } \ T_k=e^{\theta_k J}\overline{T},\]
with $\theta_k\in\R$ given by $\tau+iH_k=ie^{i\theta_k}\overline{H}$. Consequently, $\Sigma$ is part of $\Sigma_1$ or of $\Sigma_2$.

Assume $H\neq 0$. Then $\Sigma_1$ and $\Sigma_2$ are twin surfaces \cite[Theorem 5.14]{danielcmh07}. On the other hand, up to an orientation reversing isometry, Example \ref{helicoid-psl} provided two non-congruent $H$-constant mean curvature surfaces in $\ekappatau$ with constant intrinsic curtavure $K=4H^2+\kappa<0$ and non-constant angle function. Then, these two surfaces coincide with $\Sigma_1$ and $\Sigma_2$, and we conclude that $\Sigma$ is part of one of the surfaces of Example \ref{helicoid-psl}.

Assume $H=0$. Then $\Sigma_1=\Sigma_2$ (since they have the same Gauss-Codazzi data). On the other hand, Example \ref{parabolicsurface-psl} provided a minimal surface in $\ekappatau$ with constant intrinsic curvature $K=\kappa<0$ and non-constant angle function. Then, this surface coincide with $\Sigma_1$, and we conclude that $\Sigma$ is part of one of the surfaces of Example \ref{parabolicsurface-psl}.

\end{proof}

\section{Parallel mean curvature surfaces in $\s^2\times\s^2$ and $\h^2\times\h^2$ with constant intrinsic curvature} \label{sectionpmc}

\subsection{Preliminaries and first examples} \label{prelim}

A parallel mean curvature surface in $\M^2_c\times\M^2_c$ is a surface whose mean curvature vector $\mathcal{H}$ satisfies $$\nabla^\perp\mathcal{H}=0,$$ where $\nabla^\perp$ is the normal connection. This implies in particular that the norm of $\mathcal{H}$ is constant on the surface. For $H>0$ we will call $H$-PMC surface a parallel mean curvature surface whose mean curvature vector has norm $H$. An $H$-PMC immersion is an immersion whose image is an $H$-PMC surface. We will not consider minimal surfaces in this section.

Torralbo and Urbano \cite{torralbotams12} studied PMC surfaces in $\M^2_c\times\M^2_c$. In particular (Theorem 1) they proved that there is a one-to-one correspondence between classes of congruence of $H$-PMC immersions into $\M^2_c\times\M^2_c$ and unordered pairs of classes of congruence of $H$-CMC immersions into $\M^2_c\times\R$ (we say that two immersions $\varphi_1$ and $\varphi_2$ into some Riemannian manifold $M$ are congruent if there exists an isometry $f$ of $M$ such that $\varphi_2=f\circ\varphi_1$). In particular, from \cite[Theorem 1]{torralbotams12} and our Theorem \ref{theo1} follows that a class of congruence of an $H$-PMC immersion of an oriented surface $\Sigma$ in $\M^2_c\times\M^2_c$ can be characterized by a metric $\dif s^2$ and an unordered pair $\{\pm(\nu_1,\dif h_1),\pm(\nu_2,\dif h_2)\}$ where $(\nu_1,h_1)$ and $(\nu_2,h_2)$ are solutions to \eqref{M1}, \eqref{M2}, \eqref{M3} and \eqref{M4}. We refer to Section \ref{rktu} for details.

Up to scaling, we may assume that $c=\pm1$.

 The product manifold $\s^2\times\R$ is isometrically immersed into $\s^2\times\s^2$ as a totally geodesic $\s^2\times\s^1$, $\s^1$ being a geodesic of $\s^2$; similarly, $\h^2\times\R$ is isometrically embedded into $\h^2\times\h^2$ as a totally geodesic $\h^2\times\R$, $\R$ being a geodesic of $\h^2$. Hence, for $H\neq0$, an $H$-CMC surface in $\s^2\times\R$ or $\h^2\times\R$ yields via these immersions an $|H|$-PMC surface in $\s^2\times\s^2$ or $\h^2\times\h^2$. Moreover these immersions are characterized as follows: $[\Phi_1]=[\Phi_2]$ if and only if $\Phi$ is a CMC immersion into  a totally geodesic $\s^2\times\s^1$ or a totally geodesic $\h^2\times\R$. Then, in this case we have $[\Phi]=[\Phi_1]$. For this we refer to \cite[page 788 and Theorem 1]{torralbotams12}.

Using \cite[Theorem 1]{torralbotams12} and our Theorem \ref{theo-Kconstante}, we will classify $H$-PMC surfaces in $\M^2_c\times\M^2_c$ with constant intrinsic curvature. We have to be careful to the fact that \cite[Theorem 1]{torralbotams12} is a correspondence between immersions and not surfaces.

We first recall examples from \cite{torralbotams12}.

\begin{example}[CMC surfaces in totally geodesic hypersurfaces, see {\cite[page 788]{torralbotams12}}] \label{ex1} Since $\M^2_c\times\R$ is isometrically immersed as a totally geodesic submanifold of $\M^2_c\times\M^2_c$ (as mentioned above), all $H$-CMC surfaces with constant intrinsic curvature given by Theorem \ref{theo-Kconstante} yield $H$-PMC surfaces in $\M^2_c\times\M^2_c$ with constant intrinsic curvature.
\end{example}

\begin{example}[products of constant curvature curves, see {\cite[page 789 and Example 1]{torralbotams12}}] If $\gamma_1$ and $\gamma_2$ are curves of constant curvatures $k_1$ and $k_2$ in $\M^2_c$ such that $k_1$ and $k_2$ are not both zero, then $\gamma_1\times\gamma_2$ is an $H$-PMC surface with $4H^2=k_1^2+k_2^2$. In particular, when $\gamma_1$ or $\gamma_2$ is a geodesic of $\M^2_c$, then this PMC surface is also one of those explained in Example \ref{ex1}.
\end{example}

\begin{example}[Torralbo-Urbano surface in $\h^2\times\h^2$, see {\cite[item 3 of Theorem 4 and Remark 6]{torralbotams12}}] Here we set $c=-1$. For $H\in(0,\frac12)$, Torralbo and Urbano introduced an explicit $H$-PMC isometric embedding of the hyperbolic plane of curvature $K=4H^2-1$ that does not lie in a totally geodesic $\h^2\times\R$.
\end{example}

\subsection{New examples in $\h^2\times\h^2$}

We now describe new examples of $H$-PMC immersions into $\h^2\times\h^2$ with $H\in(0,\frac12)$ and intrinsic curvature $K=4H^2-1<0$.

We set $\Omega=\{z\in\C:\im z>0\}$ and we endow $\Omega$ with the metric $$\dif s^2=\frac{4|\dif z|^2}{K(z-\bar z)^2}.$$ This is the half-plane model for the hyperbolic plane of curvature $K$.
We assume that $\Omega$ is oriented.
 Its isometry group is $$\iso\Omega=\mathrm{PSL}_2(\R)\cup\mathrm{PSL}^-_2(\R),$$
where $\mathrm{PSL}^-_2(\R)=\{f\in\mathrm{GL}_2(\R):\det f=-1\}/\{\pm\mathrm{id}\}$, with $\mathrm{PSL}_2(\R)$ (group of orientation preserving isometries) acting on $\Omega$ by
 $$\left(\begin{array}{cc} \alpha & \beta \\ \gamma & \delta \end{array}\right)(z)=\frac{\alpha z+\beta}{\gamma z+\delta}$$
  and $\mathrm{PSL}^-_2(\R)$ (set of orientation reversing isometries) acting on $\Omega$ by $$\left(\begin{array}{cc} \alpha & \beta \\ \gamma & \delta \end{array}\right)(z)=\frac{\alpha\bar z+\beta}{\gamma\bar z+\delta}.$$ We also set $$D=\left\{\left(\begin{array}{cc} \alpha & 0 \\ 0 & \alpha^{-1} \end{array}\right):\alpha\in\R^*\right\},\quad
T=\left\{\left(\begin{array}{cc} \alpha & \beta \\ 0 & \alpha^{-1} \end{array}\right):\alpha\in\R^*,\beta\in\R\right\},$$
$$\xi=\left(\begin{array}{cc} 0 & 1 \\ -1 & 0 \end{array}\right),\quad
\eta=\left(\begin{array}{cc} -1 & 0 \\ 0 & 1 \end{array}\right),\quad
\zeta=\left(\begin{array}{cc} 0 & 1 \\ -1 & 1 \end{array}\right).$$ i.e., $$\xi(z)=-\frac1z,\quad\eta(z)=-\bar z,\quad\zeta(z)=\frac1{-z+1}.$$
For $f=\left(\begin{array}{cc} \alpha & \beta \\ \gamma & \delta \end{array}\right)\in\mathrm{PSL}_2(\R)$ we set $$\rho(f)=\frac{\beta\gamma}{\alpha\delta}\in\R\cup\{\infty\}.$$

Finally, for an immersion $\Phi:\Omega\to\h^2\times\R$ or $\Phi:\Omega\to\h^2\times\h^2$ we let $$G_\Phi=\{f\in\iso\Omega:[\Phi\circ f]=[\Phi]\}.$$ This is a subgroup of $\iso\Omega$.

Obviously classes of congruence of $H$-CMC isometric immersions of $\Omega$ into $\M^2_c\times\R$ are identical to classes of congruence of $-H$-CMC isometric immersions of $\Omega$ into $\M^2_c\times\R$ because of reflections with respect to a totally geodesic surface; such a reflection multiplies either $\nu$ or $\dif h$ by $-1$, see Remark \ref{rmksign}.  For $\pm H$-CMC isometric immersions $\Phi_j:\Omega\to\h^2\times\R$ ($j=1,2$) of data $(\nu_j,h_j)$, we have $[\Phi_1]=[\Phi_2]$ if and only if $\nu_2=\pm\nu_1$ and $\dif h_2=\pm\dif h_1$. Moreover, if $\Phi_2=\Phi_1\circ f$ for some $f\in\iso\Omega$, then $h_2=h_1\circ f$ and $\nu_2=\pm\nu_1\circ f$ according to whether $f$ preserves or reverses orientation. Hence $[\Phi_1\circ f]=[\Phi_1]$ if and only if $\nu_1\circ f=\pm\nu_1$ and $h_1\circ f=\pm h_1$.

We let $\mathcal{X}:\Omega\to\h^2\times\R$ be the ARL-immersion given in Example \ref{exarlsurf} and $\mathcal{Y}:\Omega\to\h^2\times\R$ the helicoidal immersion of Example \ref{helicoid}.  We first determine the groups $G_\mathcal{X}$ and $G_\mathcal{Y}$.

\begin{lemma} We have
\begin{eqnarray*}
 G_\mathcal{X} & = & T\cup T\eta \\
 & = & \left\{\left(\begin{array}{cc} \alpha & \beta \\ 0 & \alpha^{-1} \end{array}\right):\alpha\in\R^*,\beta\in\R\right\}
\cup\left\{\left(\begin{array}{cc} -\alpha & \beta \\ 0 & \alpha^{-1} \end{array}\right):\alpha\in\R^*,\beta\in\R\right\}.
\end{eqnarray*}
\end{lemma}

\begin{proof}
 We recall from Example \ref{exarlsurf} that the angle function of $\mathcal{X}$ is constant and its height function is $$h(z)=-\frac{2H}{\sqrt{-K}}\log\left(\frac{\sqrt{-K}}{2i}(z-\bar z)\right),$$ so $$h_z(z)=-\frac{2H}{\sqrt{-K}(z-\bar z)}.$$

 We first consider an isometry $f=\left(\begin{array}{cc} \alpha & \beta \\ \gamma & \delta \end{array}\right)\in\mathrm{PSL}_2(\R)$. The height function of $\mathcal{X}\circ f$ is $h\circ f$ and we have, since $f$ is holomorphic,
 $$(h\circ f)_z(z)=(h_z\circ f)(z)f'(z)=-\frac{2H}{\sqrt{-K}}\frac{f'(z)}{f(z)-\overline{f(z)}}=-\frac{2H}{\sqrt{-K}(z-\bar z)}\frac{\gamma\bar z+\delta}{\gamma z+\delta}.$$
 Since $\nu$ is constant, we have $[\mathcal{X}\circ f]=[\mathcal{X}]$ if and only if $(h_z\circ f)f'=\pm h_z$, i.e., if and only if $\gamma=0$. This proves that $G_\mathcal{X}\cap\mathrm{PSL}_2(\R)=T$.

 Moreover, we have $h\circ\eta=h$, so $\eta\in G_\mathcal{X}$ and $G_\mathcal{X}=T\cup T\eta$.
\end{proof}

\begin{lemma} We have
$$G_\mathcal{Y}=D\cup D\xi=\left\{\left(\begin{array}{cc} \alpha & 0 \\ 0 & \alpha^{-1} \end{array}\right):\alpha\in\R^*\right\}\cup
\left\{\left(\begin{array}{cc} 0 & \beta \\ -\beta^{-1} & 0 \end{array}\right):\beta\in\R^*\right\}.$$
\end{lemma}

\begin{proof}
We recall from Example \ref{helicoid} that the angle function of $\mathcal{Y}$ is $$\nu(z)=\frac{\sqrt{-K}}2\frac{z+\bar z}{|z|}$$ and its height function is
$$h(z)=\frac{2H}{\sqrt{-K}}\arcsinh\frac{i(z+\bar z)}{z-\bar z}-\log|z|,$$
so $$h_z(z)=-\frac{2H}{\sqrt{-K}}\frac{\bar z}{|z|(z-\bar z)}-\frac1{2z}.$$

We first consider an isometry $f=\left(\begin{array}{cc} \alpha & \beta \\ \gamma & \delta \end{array}\right)\in\mathrm{PSL}_2(\R)$. Then $$(\nu\circ f)(z)=\frac{\sqrt{-K}}2\frac{2\alpha\gamma|z|^2+(\alpha\delta+\beta\gamma)(z+\bar z)+2\beta\delta}{|\alpha z+\beta||\gamma z+\delta|}.$$ If $f\in G_\mathcal{Y}$, then $\nu\circ f=\pm\nu$ and, considering the points where these functions vanish, we get $\alpha\gamma=\beta\delta=0$, so $f\in D\cup D\xi$.

Conversely, we check that if $f\in D$ then $\nu\circ f=\nu$ and $(h\circ f)_z=h_z$, and if $f\in D\xi$ then $\nu\circ f=-\nu$ and $(h\circ f)_z=-h_z$, so $[\mathcal{Y}\circ f]=[\mathcal{Y}]$ in both cases.

This proves that $G_\mathcal{Y}\cap\mathrm{PSL}_2(\R)=D\cup D\xi$.

We now consider an isometry $f=\left(\begin{array}{cc} \alpha & \beta \\ \gamma & \delta \end{array}\right)\in\mathrm{PSL}^-_2(\R)$. Assume that $f\in G_\mathcal{Y}$. Then $\nu\circ f=\pm\nu$, and, since $\nu\circ\eta=-\nu$, we have $\nu\circ(f\circ\eta)=\pm\nu$. By the previous discussion we obtain that $f\circ\eta\in D\cup D\xi\subset G_\mathcal{Y}$.

Consequently we get $\eta\in G_\mathcal{Y}$. But $$(h\circ\eta)(z)=-\frac{2H}{\sqrt{-K}}\arcsinh\frac{i(z+\bar z)}{z-\bar z}-\log|z|,$$
so $$(h\circ\eta)_z(z)=\frac{2H}{\sqrt{-K}}\frac{\bar z}{|z|(z-\bar z)}-\frac1{2z},$$
so $(h\circ\eta)_z\neq\pm h_z$, which gives a contradiction.

Hence $G_\mathcal{Y}\cap\mathrm{PSL}^-_2(\R)=\emptyset$.
\end{proof}

\begin{definition}
 For $f\in\iso\Omega$, we let $\mathcal{A}_f:\Omega\to\h^2\times\h^2$ be an $H$-PMC immersion corresponding to the pair $([\mathcal{Y}],[\mathcal{X}\circ f])$; this immersion is unique up to congruences in $\h^2\times\h^2$.

 For $f\in\iso\Omega$, we let $\mathcal{B}_f:\Omega\to\h^2\times\h^2$ be an $H$-PMC immersion corresponding to the pair $([\mathcal{Y}],[\mathcal{Y}\circ f])$; this immersion is unique up to congruences in $\h^2\times\h^2$.
 \end{definition}

 \begin{remark}
  We note that if $f\in G_\mathcal{Y}$, then $[\mathcal{Y}]=[\mathcal{Y}\circ f]$ and so, by the discussion in \cite{torralbotams12} recalled in Section \ref{prelim}, $\mathcal{B}_f$ is a CMC immersion into a totally geodesic $\h^2\times\R$.
 \end{remark}

\begin{proposition} \label{equivxy}
If $f_1,f_2\in\iso\Omega$, then the immersions $\mathcal{A}_{f_1}$ and $\mathcal{A}_{f_2}$ have the same image in $\h^2\times\h^2$ up to congruences if and only if
there exist $k\in G_\mathcal{X}$ and $g\in G_\mathcal{Y}$ such that $$f_2=k\circ f_1\circ g.$$

Moreover, this property defines on $\iso\Omega$ and equivalence relation $\sim$ whose equivalence classes are $G_\mathcal{X} G_\mathcal{Y}$ (the class of $\mathrm{id}$) and $\iso\Omega\setminus G_\mathcal{X} G_\mathcal{Y}$ (the class of $\zeta$).
\end{proposition}

\begin{proof}
 The immersions $\mathcal{A}_{f_1}$ and $\mathcal{A}_{f_2}$ have the same image in $\h^2\times\h^2$ up to congruences if and only if there exists $g\in\iso\Omega$ such that $$[\mathcal{A}_{f_1}\circ g]=[\mathcal{A}_{f_2}].$$ This condition is equivalent to $$([\mathcal{Y}\circ g],[\mathcal{X}\circ f_1\circ g])=([\mathcal{Y}],[\mathcal{X}\circ f_2]),$$
 i.e., to $$[\mathcal{Y}\circ g]=[\mathcal{Y}]\quad\textrm{and}\quad[\mathcal{X}\circ f_1\circ g\circ f_2^{-1}]=[\mathcal{X}],$$
 i.e., to $$g\in G_\mathcal{Y}\quad\textrm{and}\quad f_1\circ g\circ f_2^{-1}\in G_\mathcal{X}.$$ Setting $k=(f_1\circ g\circ f_2^{-1})^{-1}$, this proves the first assertion.

 We now study the equivalence classes of $\sim$.


It is clear that $G_\mathcal{X} G_\mathcal{Y}$ is the equivalence class of $\mathrm{id}$.

We also notice that
$$G_\mathcal{X}G_\mathcal{Y}=G_\mathcal{X}\cup G_\mathcal{X}\xi=\left\{\left(\begin{array}{cc} \alpha & \beta \\ \gamma & \delta \end{array}\right)\in\iso\Omega:\gamma=0\textrm{ or }\delta=0\right\}.$$

Let $f=\left(\begin{array}{cc} \alpha & \beta \\ \gamma & \delta \end{array}\right)\in\iso\Omega\setminus G_\mathcal{X} G_\mathcal{Y}$. Then $\gamma\neq0$ and $\delta\neq0$.

If $f\in\mathrm{PSL}_2(\R)$ and $\frac\gamma\delta<0$ then
$$\left(\begin{array}{cc} \alpha & \beta \\ \gamma & \delta \end{array}\right)
=\left(\begin{array}{cc} \frac1{\delta\mu} & -\frac\alpha\mu \\ 0 & \delta\mu \end{array}\right)
\left(\begin{array}{cc} 0 & 1 \\ -1 & 1 \end{array}\right)
\left(\begin{array}{cc} \mu & 0 \\ 0 & \mu^{-1} \end{array}\right)$$ for $\mu^2=-\frac\gamma\delta$. If $f\in\mathrm{PSL}_2(\R)$ and $\frac\gamma\delta>0$ then
$$\left(\begin{array}{cc} \alpha & \beta \\ \gamma & \delta \end{array}\right)
=\left(\begin{array}{cc} -\frac\mu\delta & -\frac\beta\mu \\ 0 & -\frac\delta\mu \end{array}\right)
\left(\begin{array}{cc} 0 & 1 \\ -1 & 1 \end{array}\right)
\left(\begin{array}{cc} 0 & \mu \\ -\mu^{-1} & 0 \end{array}\right)$$ for $\mu^2=\frac\delta\gamma$. Hence in both cases we have $f\sim\left(\begin{array}{cc} 0 & 1 \\ -1 & 1 \end{array}\right)$.

If $f\in\mathrm{PSL}^-_2(\R)$, then, since $\eta\in G_\mathcal{X}$, $\eta\circ f\in\mathrm{PSL}_2(\R)\setminus G_\mathcal{X} G_\mathcal{Y}$, so $f\sim\eta\circ f\sim\left(\begin{array}{cc} 0 & 1 \\ -1 & 1 \end{array}\right)$.

This proves that $\iso\Omega\setminus G_\mathcal{X} G_\mathcal{Y}$ is included in the equivalence class of $\left(\begin{array}{cc} 0 & 1 \\ -1 & 1 \end{array}\right)$, and so equal to it.
\end{proof}

\begin{proposition} \label{equivyy}
If $f_1,f_2\in\iso\Omega$, then the immersions $\mathcal{B}_{f_1}$ and $\mathcal{B}_{f_2}$ have the same image in $\h^2\times\h^2$ up to congruences if and only if there exist $g_1,g_2\in G_\mathcal{Y}$ such that $$f_2=g_1\circ f_1\circ g_2.$$

This property defines on $\iso\Omega$ an equivalence relation $\approx$. For $s\in\R$ we set $$m_s=\left(\begin{array}{cc} s & 1-s \\ -1 & 1 \end{array}\right).$$ Then each equivalence class has exactly one representative in the set
$$\{\mathrm{id},\eta\}\cup\{m_s:s\in[0,+\infty)\}\cup\{\eta\circ m_s:s\in[0,+\infty)\}.$$
%
%
\end{proposition}

\begin{proof}
The proof of the first assertion is analogous to that of Proposition \ref{equivxy}.

We now study the equivalence classes of $\approx$. We first notice that:
\begin{itemize}
 \item if $f_1\approx f_2$, then $\det f_1=\det f_2$,
 \item the equivalence class of $\mathrm{id}$ is $G_\mathcal{Y}$,
 \item if $f\in\mathrm{PSL}_2(\R)\setminus G_\mathcal{Y}$ and $g\in D$, then $\rho(g\circ f)=\rho(f\circ g)=\rho(f)$,
 \item if $f\in\mathrm{PSL}_2(\R)\setminus G_\mathcal{Y}$ and $g\in D\xi$, then $\rho(g\circ f)=\rho(f\circ g)=\rho(f)^{-1}$.
\end{itemize}

{\bf Claim A.} If $f=\left(\begin{array}{cc} \alpha & \beta \\ \gamma & \delta \end{array}\right)\in\mathrm{PSL}_2(\R)\setminus G_\mathcal{Y}$, then $f\approx m_s$ for some $s\in\R$.

If $\gamma\neq0$ and $\delta\neq0$, then:
\begin{itemize}
 \item if $\gamma$ and $\delta$ have the same sign, then $$\left(\begin{array}{cc} \alpha & \beta \\ \gamma & \delta \end{array}\right)
 =\left(\begin{array}{cc} -\frac\mu\delta & 0 \\ 0 & -\frac\delta\mu \end{array}\right)
 \left(\begin{array}{cc} s & 1-s \\ -1 & 1 \end{array}\right)
 \left(\begin{array}{cc} 0 & \mu \\ -\mu^{-1} & 0 \end{array}\right)$$ with $s=-\beta\gamma$ and $\mu^2=\frac\delta\gamma$,
\item if $\gamma$ and $\delta$ have opposite signs, then $$\left(\begin{array}{cc} \alpha & \beta \\ \gamma & \delta \end{array}\right)
 =\left(\begin{array}{cc} \frac1{\delta\mu} & 0 \\ 0 & \delta\mu \end{array}\right)
 \left(\begin{array}{cc} s & 1-s \\ -1 & 1 \end{array}\right)
 \left(\begin{array}{cc} \mu & 0 \\ 0 & \mu^{-1} \end{array}\right)$$ with $s=\alpha\delta$ and $\mu^2=-\frac\gamma\delta$.
\end{itemize}
Hence, in both cases, $f\approx m_s$ for some $s\in\R$.

 If $\gamma=0$, then $\alpha\neq0$ and, since $f\notin G_{\mathcal{Y}}$, $\beta\neq 0$. Similarly, if $\delta=0$, then $\beta\neq0$ and, since $f\notin G_\mathcal{Y}$, $\alpha\neq 0$. Hence, in both cases, applying the previous argument to $\xi\circ f=\left(\begin{array}{cc} \gamma & \delta \\ -\alpha & -\beta \end{array}\right)$ we obtain that $f\approx\xi\circ f\approx m_s$ for some $s\in\R$.

 This proves Claim A.

 {\bf Claim B.} Let $s,t\in\R$. Then $m_s\approx m_t$ if and only if
 \begin{itemize}
  \item $s=t$,
  \item or $s+t=1$ and $s\notin[0,1]$.
 \end{itemize}

 First, if $s>1$ or $s<0$, then $-s$ and $s-1$ do not vanish and have opposite signs, so by the previous arguments we have
 $$m_s\approx\xi\circ m_s=\left(\begin{array}{cc} -1 & 1 \\ -s & s-1 \end{array}\right)\approx m_{1-s}.$$

 Conversely, let $s$ and $t$ such that $m_s\approx m_t$ and $s\neq t$. Then $\rho(m_s)=\rho(m_t)$ or $\rho(m_s)=\rho(m_t)^{-1}$. In the first case we have $\frac s{s-1}=\frac t{t-1}$, so $s=t$. In the second case we have $\frac s{s-1}=\frac{t-1}t$, so $t=1-s$.

 Assume that $s\in[0,1]$. Then $s\geqslant0$ and $1-s\geqslant0$. Let $g_1,g_2\in G_{\mathcal{Y}}$ such that $m_{1-s}=g_1\circ m_s\circ g_2$. Considering the signs of the coefficients, both $m_s$ and $m_{1-s}$ are of type $\pm\left(\begin{array}{cc} \geqslant0 & \geqslant0 \\ <0 & >0 \end{array}\right)$. Since $\rho(m_{1-s})=\rho(m_s)^{-1}$, we have $g_1\in D$ and $g_2\in D\xi$ or the contrary. But in the first case we obtain that $m_{1-s}$ is of type $\pm\left(\begin{array}{cc} \leqslant0 & \geqslant0 \\ <0 & <0 \end{array}\right)$ and in the second case of type $\pm\left(\begin{array}{cc} <0 & >0 \\ \leqslant0 & \leqslant0 \end{array}\right)$, which is a contradiction in both cases. Hence we cannot have $m_{1-s}\approx m_s$ if $s\in[0,1]$.

 This proves Claim B.

 {\bf Claim C.} If $f\in\mathrm{PSL}_2(\R)\setminus G_\mathcal{Y}$, then there exists a unique $s\in[0,+\infty)$ such that $f\approx m_s$.

 By Claim A, there exists $t\in\R$ such that $f\approx m_t$. If $t\geqslant0$, set $s=t$; if $t<0$, set $s=1-t$. Then in both cases we have $s\in[0,+\infty)$ and $f\approx m_t\approx m_s$ by Claim B.

 And if there is $u\in[0,+\infty)$ such that $u\neq s$ and $f\approx m_u$, then by Claim B we have $s+u=1$ and $s>1$, so $u<0$, which is a contradiction.

 This proves Claim C.

 {\bf Claim D.} Let $f_1,f_2\in\iso\Omega$. Then $f_1\approx f_2$ if and only if $\eta\circ f_1\approx\eta\circ f_2$.

 This is a consequence of the fact that $\eta$ commutes with all elements of $G_\mathcal{Y}$.

 {\bf Claim E.} If $f\in\mathrm{PSL}^-_2(\R)$, then either $f\approx\eta$ or there exists a unique $s\in[0,+\infty)$ such that $f\approx\eta\circ m_s$.

 This is a consequence of Claim D and the fact that, either $\eta\circ f\in G_\mathcal{Y}$ and then $\eta\circ f\approx\mathrm{id}$, or $\eta\circ f\in\mathrm{PSL}_2(\R)\setminus G_\mathcal{Y}$ and then we conclude by Claim C.
\end{proof}

We now determine, for each of these immersions, the groups of isometries of $\Omega$ that are induced by ambient isometries.
\begin{proposition} \label{groups}
 We have $$G_{\mathcal{A}_\mathrm{id}}=D,$$
  $$G_{\mathcal{A}_\zeta}=\{\mathrm{id}\},$$
    $$G_{\mathcal{B}_\eta}=G_\mathcal{Y},$$
   $$G_{\mathcal{B}_{m_s}}=G_{\mathcal{B}_{\eta\circ m_s}}=\{\mathrm{id}\}\textrm{ if }s\notin(0,1),$$
   $$G_{\mathcal{B}_{m_s}}=G_{\mathcal{B}_{\eta\circ m_s}}
   =\left\{\mathrm{id},\left(\begin{array}{cc} 0 & \sqrt{\frac{1-s}s} \ \\ -\sqrt{\frac s{1-s}} & 0 \end{array}\right)\right\}\textrm{ if }s\in(0,1).$$
\end{proposition}

\begin{proof}
Let $g\in\iso\Omega$.

We have $[\mathcal{A}_\mathrm{id}\circ g]=[\mathcal{A}_\mathrm{id}]$ if and only if $$([\mathcal{Y}\circ g],[\mathcal{X}\circ g])=([\mathcal{Y}],[\mathcal{X}]),$$ i.e., if and only if $g\in G_\mathcal{Y}\cap G_\mathcal{X}=D$. Hence $G_{\mathcal{A}_\mathrm{id}}=D$.

We have $[\mathcal{A}_\zeta\circ g]=[\mathcal{A}_\zeta]$ if and only if $$([\mathcal{Y}\circ g],[\mathcal{X}\circ\zeta\circ g])=([\mathcal{Y}],[\mathcal{X}\circ\zeta]),$$ i.e., if and only if $g\in G_\mathcal{Y}$ and $\hat g=\zeta\circ g\circ\zeta^{-1}\in G_\mathcal{X}$. Considering the eigenvalues of $g$ and $\hat g$, we obtain that $g=\left(\begin{array}{cc} \lambda & 0 \\ 0 & \lambda^{-1} \end{array}\right)$ for some $\lambda\in\R^*$ and $\hat g=\left(\begin{array}{cc} \alpha & \beta \\ 0 & \alpha^{-1} \end{array}\right)$ for $\alpha=\lambda$ or $\lambda^{-1}$ and some $\beta\in\R$. Reporting in the relation $\hat g\circ\zeta=\zeta\circ g$ we get $\lambda=\pm1$ and $\beta=0$, hence $g=\hat g=\mathrm{id}$. This proves that $G_{\mathcal{A}_\zeta}=\{\mathrm{id}\}$.

We have $[\mathcal{B}_{\eta}\circ g]=[\mathcal{B}_{\eta}]$ if and only if $$([\mathcal{Y}\circ g],[\mathcal{Y}\circ\eta\circ g])=([\mathcal{Y}],[\mathcal{Y}\circ\eta]),$$ i.e., if and only if $g\in G_\mathcal{Y}$ and $\eta\circ g\circ\eta\in G_\mathcal{Y}$. Since $\eta$ commutes with elements in $G_\mathcal{Y}$, this condition is equivalent to
$g\in G_\mathcal{Y}$. This proves that $G_{\mathcal{B}_\eta}=G_\mathcal{Y}$.

Let $s\in\R$. We have $[\mathcal{B}_{m_s}\circ g]=[\mathcal{B}_{m_s}]$ if and only if $$([\mathcal{Y}\circ g],[\mathcal{Y}\circ m_s\circ g])=([\mathcal{Y}],[\mathcal{Y}\circ m_s]),$$ i.e., if and only if $g\in G_\mathcal{Y}$ and $\hat g=m_s\circ g\circ {m_s}^{-1}\in G_\mathcal{Y}$. Considering the eigenvalues of $g$ and $\hat g$ we can see that there are two cases:
 \begin{itemize}
  \item either $g\in D$ and $\hat g=g$ or $g^{-1}$, and then we can compute that the only possibility is $g=\hat g=\mathrm{id}$,
\item or $g\in D\xi$ and $\hat g\in D\xi$, and then we can compute that the only possibility is $s\in(0,1)$, $g=\left(\begin{array}{cc} 0 & \mu \\ -\mu^{-1} & 0 \end{array}\right)$ with $\mu^2=\frac{1-s}s$ and $\hat g=\left(\begin{array}{cc} 0 & \mu s \\ -\mu^{-1}s^{-1} & 0 \end{array}\right)$.
 \end{itemize}
 This proves the claimed assertions for $G_{\mathcal{B}_{m_s}}$.

 Similarly, we have $[\mathcal{B}_{\eta\circ m_s}\circ g]=[\mathcal{B}_{\eta\circ m_s}]$ if and only if $g\in G_\mathcal{Y}$ and $\hat g=(\eta\circ m_s)\circ g\circ {(\eta\circ m_s)}^{-1}\in G_\mathcal{Y}$. Since $\eta^2=\mathrm{id}$ and elements in $G_\mathcal{Y}$ commute with $\eta$, this is equivalent to $g\in G_\mathcal{Y}$ and $m_s\circ g\circ {m_s}^{-1}\in G_\mathcal{Y}$. Hence $G_{\mathcal{B}_{\eta\circ m_s}}=G_{\mathcal{B}_{m_s}}$.
\end{proof}

\subsection{Classification}

\begin{theorem} \label{thms2s2}
 Let $H>0$ and $K\in\R$. Let $\Sigma$ be an $H$-PMC surface with constant intrinsic curvature $K$ in $\s^2\times\s^2$. Then $K=0$ and $\Sigma$ is part of a product of two curves of constant curvatures.
\end{theorem}

\begin{proof}
 Let $U$ be a surface with a Riemannian metric and $\Phi:U\to\s^2\times\s^2$ be an isometric immersion such that $\Phi(U)=\Sigma$. Let $\Phi_1,\Phi_2:U\to\s^2\times\R$ be the two $H$-CMC isometric immersions given by the Torralbo-Urbano correspondence. Then, by Theorem \ref{theo-Kconstante} for $c=1$, we have $K=0$ and $\Phi_1(U)$ and $\Phi_2(U)$ are parts of vertical cylinders. In particular their angle functions vanish.

But the angles functions of $\Phi_1$ and $\Phi_2$ coincide with the K\"ahler functions of $\Phi$ for the two K\"ahler structures of $\s^2\times\s^2$ (see  \cite[page 790 and the proof of Theorem 1]{torralbotams12}, precisely the relation $C_j=\nu_j$). Then, by \cite[Theorem 2]{torralbotams12}, $\Sigma$ is part of a product of two curves of constant curvatures.
\end{proof}

\begin{theorem} \label{thmh2h2}
 Let $H>0$ and $K\in\R$. Let $\Sigma$ be an $H$-PMC surface with constant intrinsic curvature $K$ in $\h^2\times\h^2$. Then
  \begin{itemize}
   \item either $K=0$ and and $\Sigma$ is part of a product of two curves of constant curvatures,
   \item either $K=4H^2-1<0$ and $\Sigma$ is part of
   \begin{itemize}
    \item an ARL-surface in some totally geodesic $\h^2\times\R$,
    \item a helicoidal surface of Example \ref{helicoid} in some totally geodesic $\h^2\times\R$,
    \item a Torralbo-Urbano surface,
    \item a surface $\mathcal{A}_\mathrm{id}(\Omega)$,
    \item a surface $\mathcal{A}_\zeta(\Omega)$,
     \item a surface $\mathcal{B}_\eta(\Omega)$,
    \item a surface $\mathcal{B}_{m_s}(\Omega)$ for some $s\in[0,+\infty)$,
    \item or a surface $\mathcal{B}_{\eta\circ m_s}(\Omega)$ for some $s\in[0,+\infty)$.
   \end{itemize}
  \end{itemize}
Moreover, the surfaces appearing in this list are pairwise non congruent.
\end{theorem}

\begin{proof}
 Let $U$ be a surface with a Riemannian metric and $\Phi:U\to\h^2\times\h^2$ be an isometric immersion such that $\Phi(U)=\Sigma$. Let $\Phi_1,\Phi_2:U\to\h^2\times\R$ be the two $H$-CMC isometric immersions given by the Torralbo-Urbano correspondence. Then, by Theorem \ref{theo-Kconstante} for $c=-1$,
 \begin{itemize}
  \item either $K=0$ and $\Phi_1(U)$ and $\Phi_2(U)$ are parts of vertical cylinders,
  \item or $K=4H^2-1$ and $\Phi_1(U)$ and $\Phi_2(U)$ are each part of an ARL-surface or of a helicoidal surface of Example \ref{helicoid}
 \end{itemize}

In the first case, we conclude as in the proof of Theorem \ref{thms2s2} that $\Sigma$ is part of a product of two curves of constant curvatures.

From now on we assume that we are in the second case. Then, up to a repara-metrization, we can assume that $U$ is an open set of $\Omega$ and that there exist $f_1,f_2\in\iso\Omega$ such that $\Phi_1=\mathcal{X}\circ f_1$ or $\Phi_1=\mathcal{Y}\circ f_1$ and $\Phi_2=\mathcal{X}\circ f_2$ or $\Phi_2=\mathcal{Y}\circ f_2$. In all cases $\Phi(U)$ will be part of a complete surface so it is not restrictive to assume that $U=\Omega$.

Assume that $\Phi_1=\mathcal{X}\circ f_1$ and $\Phi_2=\mathcal{X}\circ f_2$. The PMC immersion $\Phi$ admits two holomorphic \emph{Hopf differentials}, and they coincide, up to the multiplication by a constant, with the Abresch-Rosenberg differentials of $\Phi_1$ and $\Phi_2$ \cite[Definition 1, Proposition 3 and Theorem 1]{torralbotams12}. Then, as the ARL-surface has a vanishing Abresch-Rosenberg differential, the Hopf differentials of $\Phi$ vanish. Consequently, by \cite[Theorem 4]{torralbotams12}, either $\Sigma$ is an $H$-CMC surface in a totally geodesic $\h^2\times\R$ with vanishing Abresch-Rosenberg differential, so an ARL-surface because it also has constant curvature, or $\Sigma$ is a Torralbo-Urbano surface.

Assume that $\Phi_1=\mathcal{Y}\circ f_1$ and $\Phi_2=\mathcal{X}\circ f_2$. Then the pair $([\Phi_1\circ f_1^{-1}],[\Phi_2\circ f_1^{-1}])$ corresponds to $[\Phi\circ f_1^{-1}]$, so will yield a reparametrization of $\Sigma$ (up to congruences) in $\h^2\times\h^2$. Hence we may assume that $f_1=\mathrm{id}$. Then $\Phi=\mathcal{A}_{f_2}$ (up to congruences) and, by Proposition \ref{equivxy}, $\Sigma$ is congruent to $\mathcal{A}_\mathrm{id}(\Omega)$ or to $\mathcal{A}_\zeta(\Omega)$.

Assume that $\Phi_1=\mathcal{X}\circ f_1$ and $\Phi_2=\mathcal{Y}\circ f_2$. Then, $\Phi$ is congruent to the PMC immersion corresponding to the pair $([\Phi_2],[\Phi_1])$, the conclusion is the same as in the previous case.

Finally, assume that $\Phi_1=\mathcal{X}\circ f_1$ and $\Phi_2=\mathcal{Y}\circ f_2$. As above we can assume that $f_1=\mathrm{id}$. If $f_2\in G_\mathcal{Y}$, then $[\Phi_1]=[\Phi_2]$ and $\Sigma$ is a helicoidal surface of Example \ref{helicoid} in some totally geodesic $\h^2\times\R$. If $f_2\notin G_\mathcal{Y}$, then $\Phi=\mathcal{B}_{f_2}$ (up to congruences) and, by Proposition \ref{equivyy}, $\Sigma$ is congruent to $\mathcal{B}_\eta(\Omega)$, to $\mathcal{B}_{m_s}(\Omega)$ or to $\mathcal{B}_{\eta\circ m_s}(\Omega)$ for some $s\in[0,+\infty)$.

The fact that all surfaces in this list are pairwise non congruent is a consequence of the fact that any two of them give different pairs of classes of congruences of CMC isometric immersions into $\h^2\times\R$.
\end{proof}

\begin{remark}
The extrinsic normal curvature of an $H$-PMC surface in $\h^2\times\h^2$ is given by $$\overline K^\perp=\frac{\nu_2^2-\nu_1^2}2$$ (see \cite[page 790]{torralbotams12} and recall that $C_j=\nu_j$). Hence the surfaces $\mathcal{A}_\mathrm{id}(\Omega)$, $\mathcal{A}_\zeta(\Omega)$, $\mathcal{B}_{m_s}(\Omega)$ and $\mathcal{B}_{\eta\circ m_s}(\Omega)$ for  $s\in[0,+\infty)$ are examples of $H$-PMC surfaces in $\h^2\times\h^2$ with non identically zero extrinsic normal curvature. Up to our knowledge, these are the first examples of such surfaces. Note that \cite[Theorem 3]{torralbotams12} classifies $H$-PMC surfaces in $\s^2\times\s^2$ and $\h^2\times\h^2$ whose extrinsic normal curvature is identically zero.
\end{remark}

\begin{remark}
By Proposition \ref{groups}, the group of isometries of $\mathcal{A}_\mathrm{id}(\Omega)$ induced by ambient isometries is a one-parameter group (namely, $G_{\mathcal{A}_\mathrm{id}(\Omega)}=D$); hence, it should be possible to compute an explicit expression of this surface, possibly in terms of solutions of ordinary differential equations. The same holds for the surface $\mathcal{B}_\eta(\Omega)$. However, for the surfaces $\mathcal{A}_\zeta(\Omega)$, $\mathcal{B}_{m_s}(\Omega)$ and $\mathcal{B}_{\eta\circ m_s}(\Omega)$ for $s\in[0,+\infty)$, the groups of isometries induced by ambient isometries are discrete, so such an explicit expression may not exist.
\end{remark}

\subsection{A remark on Torralbo and Urbano's correspondence} \label{rktu}

Let $H>0$. Let $(\Sigma,\dif s^2)$ be an oriented simply connected Riemannian surface. In this section we explain in details why Torralbo and Urbano's correspondence is a one-to-one correspondence between classes of congruence of $H$-PMC immersions into $\M^2_c\times\M^2_c$ and \emph{unordered} pairs of classes of congruence of $H$-CMC immersions into $\M^2_c\times\R$.

It follows from their work and our Theorem \ref{theo1} that a $4$-uple $$(\nu_1,\dif h_1,\nu_2,\dif h_2)$$ where $(\nu_1,h_1)$ and $(\nu_2,h_2)$ are solutions to \eqref{M1}, \eqref{M2}, \eqref{M3} and \eqref{M4} on $\Sigma$ characterizes an $H$-PMC immersion into $\M^2_c\times\M^2_c$ modulo isometries in the connected component of the identity in $\iso{\M^2_c\times\M^2_c}$. In the notations of \cite[pages 795 and 796]{torralbotams12} we have, for $j=1,2$,
\begin{equation} \label{relcjnuj}
 C_j=\nu_j,\quad\eta_j=h_j,\quad\gamma_j=-i\sqrt2(\eta_j)_z.
\end{equation}

On the other hand, this $4$-uple characterizes an ordered pair of $H$-CMC immersions into $\M^2_c\times\R$ modulo isometries in the connected component of the identity in $\iso{\M^2_c\times\R}$ for each of them.

The isometry group of $\M^2_c\times\M^2_c$ has $8$ connected components; they can be obtained from the connected component of the identity applying the following isometries:
$$G_1:(x_1,x_2)\mapsto(F(x_1),x_2),$$
$$G_2:(x_1,x_2)\mapsto(x_1,F(x_2)),$$
$$G_3:(x_1,x_2)\mapsto(x_2,x_1),$$ where $F:\M^2_c\to\M^2_c$ is an orientation reversing isometry (i.e., antiholomorphic).

Following the notation of \cite[pages 787 and 790]{torralbotams12}, we let $J$ be a conformal structure on $\M^2_c$, $\omega$ be the corresponding K\"ahler form on $\M^2_c$, $\pi_j:\M^2_c\times\M^2_c\to\M^2_c$ the $j$-st projection ($j=1,2$), $\omega_1=\pi_1^*\omega+\pi_2^*\omega$ and $\omega_2=\pi_1^*\omega-\pi_2^*\omega$.

Let $\Phi:\Sigma\to\M^2_c\times\M^2_c$ be an $H$-PMC isometric immersion with data $$(\nu_1,\dif h_1,\nu_2,\dif h_2).$$ For $j=1,2$, let $\Phi_j:\Sigma\to\M^2_c\times\R$ be the $H$-CMC isometric immersion with data $(\nu_j,\dif h_j)$.

The manifold $\M^2_c\times\M^2_c$ has two K\"ahler structures, $J_1=(J,J)$ and $J_2=(J,-J)$. The K\"ahler functions $C_j$ ($j=1,2$) of $\Phi$ for these two K\"ahler structures are defined by $$\Phi^*\omega_j=C_j\omega_\Sigma$$ where $\omega_\Sigma$ is the area $2$-form on $\Sigma$. Moreover, by \cite[equations (3.2) and (3.3)]{torralbotams12}, the ``gamma'' functions $\gamma_j$ ($j=1,2$) satisfy
\begin{equation} \label{defgammaj}
 \gamma_j=\frac{\sqrt2}H\langle J_j\Phi_z,\mathcal{H}\rangle
\end{equation}
where $\mathcal{H}$ is the mean curvature vector of $\Phi$.

We have $F^*\omega=-\omega$, $\pi_1\circ G_1=F\circ\pi_1$ and $\pi_2\circ G_1=\pi_2$ so $$(G_1\circ\Phi)^*\omega_1=\Phi^*G_1^*\pi_1^*\omega+\Phi^*G_1^*\pi_2^*\omega
=-\Phi^*\pi_1^*\omega+\Phi^*\pi_2^*\omega=-\Phi^*\omega_2,$$ so the first K\"ahler function of $G_1\circ\Phi$ is $-C_2$. Similarly, the second K\"ahler function of $G_1\circ\Phi$ is $-C_1$. Also, we have $$\langle J_1(G_1\circ\Phi)_z,\dif G_1(\mathcal{H})\rangle=-\langle J_2\Phi_z,\mathcal{H}\rangle,$$
$$\langle J_2(G_1\circ\Phi)_z,\dif G_1(\mathcal{H})\rangle=-\langle J_1\Phi_z,\mathcal{H}\rangle,$$ so,
using \eqref{defgammaj}, we obtain that the ``gamma'' functions of $G_1\circ\Phi$ are $-\gamma_2$ and $-\gamma_1$. Hence, taking \eqref{relcjnuj} into account, we conclude that $G_1\circ\Phi$ has data $$(-\nu_2,-\dif h_2,-\nu_1,-\dif h_1).$$

Arguing in the same way, we can see that $G_2\circ\Phi$ has data $$(\nu_2,\dif h_2,\nu_1,\dif h_1).$$

Finally, $\pi_1\circ G_3=\pi_2$ and $\pi_2\circ G_3=\pi_1$, from what we deduce that the K\"ahler functions of $G_3\circ\Phi$ are $C_1$ and $-C_2$. We also check that the ``gamma'' functions of $G_3\circ\Phi$ are $\gamma_1$ and $-\gamma_2$. Hence $G_3\circ\Phi$ has data $$(\nu_1,\dif h_1,-\nu_2,-\dif h_2).$$

Hence, from the data of $\Phi$ and its compositions with isometries in the group generated by $G_1$, $G_2$ and $G_3$, we obtain the data of the pair $(\Phi_1,\Phi_2)$ and all pairs that can be obtained by possibly replacing each $\Phi_j$ by $R\circ\Phi_j$ where $R$ is a $\pi$-rotation around a horizontal geodesic in $\M^2_c\times\R$ and possibly permuting the two elements. Such a set exactly characterizes an \emph{unordered} pair of classes of congruence of $H$-CMC isometric immersions (note that orientation reversing isometries of $\M^2_c\times\R$ give rise to $-H$-CMC immersions and the corresponding data are not admissible for this correspondence).

\bibliographystyle{amsplain}
\bibliography{ddv}

\providecommand{\bysame}{\leavevmode\hbox to3em{\hrulefill}\thinspace}
\providecommand{\MR}{\relax\ifhmode\unskip\space\fi MR }
\providecommand{\MRhref}[2]{%
  \href{http://www.ams.org/mathscinet-getitem?mr=#1}{#2}
}
\providecommand{\href}[2]{#2}
\begin{thebibliography}{10}

\bibitem{abresch2004}
Uwe Abresch and Harold Rosenberg, \emph{A {H}opf differential for constant mean
  curvature surfaces in {${\bf S}^2\times{\bf R}$} and {${\bf H}^2\times{\bf
  R}$}}, Acta Math. \textbf{193} (2004), no.~2, 141--174. \MR{2134864}

\bibitem{abresch05}
\bysame, \emph{Generalized {H}opf differentials}, Mat. Contemp. \textbf{28}
  (2005), 1--28. \MR{2195187}

\bibitem{Cartan1938}
\'{E}lie Cartan, \emph{Familles de surfaces isoparam\'{e}triques dans les
  espaces \`a courbure constante}, Ann. Mat. Pura Appl. \textbf{17} (1938),
  no.~1, 177--191. \MR{1553310}

\bibitem{Chen72}
Bang-yen Chen, \emph{Minimal surfaces with constant {G}auss curvature}, Proc.
  Amer. Math. Soc. \textbf{34} (1972), 504--508. \MR{296828}

\bibitem{danielcmh07}
Beno\^{\i}t Daniel, \emph{Isometric immersions into 3-dimensional homogeneous
  manifolds}, Comment. Math. Helv. \textbf{82} (2007), no.~1, 87--131.
  \MR{2296059}

\bibitem{danieltams09}
\bysame, \emph{Isometric immersions into {$\Bbb S^n\times\Bbb R$} and {$\Bbb
  H^n\times\Bbb R$} and applications to minimal surfaces}, Trans. Amer. Math.
  Soc. \textbf{361} (2009), no.~12, 6255--6282. \MR{2538594}

\bibitem{danielimj15}
\bysame, \emph{Minimal isometric immersions into {$\Bbb S^2\times\Bbb R$} and
  {$\Bbb H^2\times\Bbb R$}}, Indiana Univ. Math. J. \textbf{64} (2015), no.~5,
  1425--1445. \MR{3418447}

\bibitem{vazquezmanzano18}
Miguel Dom\'{\i}nguez-V\'azquez and Jos\'e M.~Manzano, \emph{Isoparametric
  surfaces in $\mathbb{E}(\kappa,\tau)$-spaces},  (2018), To appear in {\it
  Ann. Sc. Norm. Super. Pisa Cl. Sci. (5)}, arXiv:1803.06154.

\bibitem{eisenhart40}
Luthur~Pfahler Eisenhart, \emph{An {I}ntroduction to {D}ifferential
  {G}eometry}, Princeton Mathematical Series, v. 3, Princeton University Press,
  Princeton, N. J., 1940. \MR{0003048}

\bibitem{RosenbergCMH11}
Jos\'{e}~M. Espinar and Harold Rosenberg, \emph{Complete constant mean
  curvature surfaces in homogeneous spaces}, Comment. Math. Helv. \textbf{86}
  (2011), no.~3, 659--674. \MR{2803856}

\bibitem{fernandezmira10}
Isabel Fern\'{a}ndez and Pablo Mira, \emph{Constant mean curvature surfaces in
  3-dimensional {T}hurston geometries}, Proceedings of the {I}nternational
  {C}ongress of {M}athematicians. {V}olume {II}, Hindustan Book Agency, New
  Delhi, 2010, pp.~830--861. \MR{2827821}

\bibitem{hirakawa06gd}
Shinya Hirakawa, \emph{Constant {G}aussian curvature surfaces with parallel
  mean curvature vector in two-dimensional complex space forms}, Geom. Dedicata
  \textbf{118} (2006), 229--244. \MR{2239458}

\bibitem{hoffman1973}
David~A. Hoffman, \emph{Surfaces of constant mean curvature in manifolds of
  constant curvature}, J. Differential Geometry \textbf{8} (1973), 161--176.
  \MR{0390973}

\bibitem{kenmotsu83}
Katsuei Kenmotsu, \emph{Minimal surfaces with constant curvature in
  {$4$}-dimensional space forms}, Proc. Amer. Math. Soc. \textbf{89} (1983),
  no.~1, 133--138. \MR{706526}

\bibitem{kenmotsu00}
Katsuei Kenmotsu and Ky\^{u}ya Masuda, \emph{On minimal surfaces of constant
  curvature in two-dimensional complex space form}, J. Reine Angew. Math.
  \textbf{523} (2000), 69--101. \MR{1762956}

\bibitem{lawson69}
H.~Blaine Lawson, Jr., \emph{Local rigidity theorems for minimal
  hypersurfaces}, Ann. of Math. (2) \textbf{89} (1969), 187--197. \MR{0238229}

\bibitem{Leite2007}
Maria~Luiza Leite, \emph{An elementary proof of the {A}bresch-{R}osenberg
  theorem on constant mean curvature immersed surfaces in {$\Bbb S^2\times\Bbb
  R$} and {$\Bbb H^2\times\Bbb R$}}, Q. J. Math. \textbf{58} (2007), no.~4,
  479--487. \MR{2371467}

\bibitem{levi-civita37}
Tullio Levi-Civita, \emph{Famiglie di superficie isoparametrische
  nell'ordinario spacio euclideo}, Atti. Accad. Naz. Lincei. Rend. Cl. Sci.
  Fis. Mat. Natur. \textbf{26} (1937), 355--362.

\bibitem{Penafiel2012}
Carlos Pe\~{n}afiel, \emph{Invariant surfaces in {$\widetilde{PSL}_2(\Bbb
  R,\tau)$} and applications}, Bull. Braz. Math. Soc. (N.S.) \textbf{43}
  (2012), no.~4, 545--578. \MR{3024070}

\bibitem{penafiel15}
\bysame, \emph{Screw motion surfaces in {$\widetilde{PSL}_2(\Bbb R,\tau)$}},
  Asian J. Math. \textbf{19} (2015), no.~2, 265--280. \MR{3337787}

\bibitem{saearp2005}
Ricardo Sa~Earp and Eric Toubiana, \emph{Screw motion surfaces in {$\Bbb
  H^2\times\Bbb R$} and {$\Bbb S^2\times\Bbb R$}}, Illinois J. Math.
  \textbf{49} (2005), no.~4, 1323--1362. \MR{2210365}

\bibitem{scott83}
Peter Scott, \emph{The geometries of {$3$}-manifolds}, Bull. London Math. Soc.
  \textbf{15} (1983), no.~5, 401--487. \MR{705527}

\bibitem{torralbotams12}
Francisco Torralbo and Francisco Urbano, \emph{Surfaces with parallel mean
  curvature vector in {${\Bbb S}^2\times{\Bbb S}^2$} and {${\Bbb
  H}^2\times{\Bbb H}^2$}}, Trans. Amer. Math. Soc. \textbf{364} (2012), no.~2,
  785--813. \MR{2846353}

\bibitem{Verpoort2014}
Steven Verpoort, \emph{Hypersurfaces with a parallel higher fundamental form},
  J. Geom. \textbf{105} (2014), no.~2, 223--242. \MR{3227482}

\bibitem{Wall1986}
C.~T.~C. Wall, \emph{Geometric structures on compact complex analytic
  surfaces}, Topology \textbf{25} (1986), no.~2, 119--153. \MR{837617}

\end{thebibliography}

\end{document}